\documentclass[review]{elsarticle}

\usepackage{hyperref}
\usepackage{amsmath,amsthm}
\usepackage{url}
\usepackage{float}
\usepackage{subfigure}
\usepackage{makecell}
\usepackage{multirow}
\newtheorem{thm}{Theorem}

\journal{}

\begin{document}

\begin{frontmatter}

\title{Traveling wave solutions for the generalized (2+1)-dimensional Kundu-Mukherjee-Naskar equation}

%
\author[mymainaddress]{Minrong Ren}

\author[mymainaddress]{YuQian Zhou\corref{mycorrespondingauthor}}
\cortext[mycorrespondingauthor]{Corresponding author}
\ead{cs97zyq@aliyun.com}
\author[mysecondaddress]{Qian Liu}
\address[mymainaddress]{College of Applied Mathematics, Chengdu University of Information Technology, Chengdu 610225, Sichuan, P. R. China}
\address[mysecondaddress]{School of Computer Science and Technology, Southwest Minzu University, Chengdu 610041, Sichuan, P. R. China}

\begin{abstract}
\par\noindent In this paper, we consider two types of traveling wave systems of the generalized Kundu-Mukherjee-Naskar equation. Firstly, due to the integrity, we obtain their energy functions. Then, the  dynamical system method is applied to study bifurcation behaviours of the two types of traveling wave systems to obtain corresponding global phase portraits  in different parameter bifurcation sets. According to them, every bounded and unbounded orbits can be identified clearly and investigated carefully which correspond to various traveling wave solutions of the generalized Kundu-Mukherjee-Naskar equation exactly. Finally, by integrating along these orbits and calculating some complicated elliptic integral, we obtain all type I and type II traveling wave solutions of the generalized Kundu-Mukherjee-Naskar equation without loss.
\end{abstract}

\begin{keyword}
GKMN equation\sep traveling wave solution\sep bifurcation \sep dynamical system
\MSC[2010] 35C07\sep35Q55\sep34C23
\end{keyword}
\end{frontmatter}
\section{Introduction}\label{sec1}
\par In this paper, we consider the following generalized (2+1)-dimensional Kundu-Mukherjee-Naskar (GKMN) equation \cite{t1}
\begin{equation}\label{1}
  iq_{t}+aq_{xy}+ibq(qq^{*}_{x}-q^{*}q_{x})=0,~~i=\sqrt{-1},
\end{equation}
where complex function $q=q(x,y,t)$ represents the profile of soliton,  $q^{*}$ is the complex conjugation of $q(x,y,t)$,  $x$, $y$ and $t$ are the spatial and temporal variables respectively. Real parameters $a$ and $b$ are the dispersion coefficient and nonlinear coefficient respectively. This equation can be used to describe optical wave propagation through coherently excited resonant waveguides, especially in the phenomenon of the bending of light beams\cite{t1a}.
\par When the coefficients $a=1$ and $b=2$, the GKMN equation degenerates to the classical KMN equation
\begin{equation}\label{2}
  iq_{t}+q_{xy}+2iq(qq^{*}_{x}-q^{*}q_{x})=0,
\end{equation}
which was first proposed by Kundu, Mukherjee and Naskar in 2014\cite{t2,t2a,t2b}. Besides description of the dynamics of optical soliton propagation in optical fibers, Eq. (\ref{2}) also can be used extensively to address the problems of oceanic rogue waves, hole waves and an ion acoustic wave in a magnetized plasma\cite{t2b,t3f,t3}.
\par Since the KMN equation is completely integrable and possesses the dynamic characteristics similar to the standard nonlinear Schr\"{o}dinger equation\cite{t2b,t3}, it can be regarded as an integrable generalization of the well-known nonlinear Schr\"{o}dinger equation
\begin{equation}\label{3}
  q_{t}+iq_{xx}+2i|q|^{2}q=0.
\end{equation}
Unlike the conventional Kerr type nonlinearity in the Eq. (\ref{3}), the nonlinear terms of Eq. (\ref{2}) can be considered as the current nonlinearity caused by chirality\cite{t3a}. Hence, Eq. (\ref{3}) only allows bright (dark) soliton for focusing (defocusing) nonlinearity or with anomalous (normal) dispersion, whereas Eq. (\ref{2}) admits both bright and dark soliton solutions regardless of its positive spatial dispersive term\cite{tm}.
\par Traveling wave solutions of the KMN equation have been always focused on by people. In 2015, Mukherjee showed the connection between Eq. (\ref{2}) and Kadomtsev Petviashvili equation and obtained its one-soliton solution, two-soliton solution and static lump solution by using Hirota bilinear method\cite{t3}. In 2017, Wen\cite{t3e} pointed out that the solutions, satisfying the focusing nonlinear Schr\"{o}dinger  equation
\begin{equation}\label{4}
  q_{y}+iq_{xx}+2i|q|^{2}q=0
\end{equation}
and the complex modified Korteweg-de Vries equation
\begin{equation}\label{5}
  q_{t}+q_{xxx}+6|q|^{2}q_{x}=0,
\end{equation}
must satisfy Eq. (\ref{2}). So, by generalizing the $n$-fold Darboux transformation of Eqs. (\ref{4}) and (\ref{5}) to the perturbation $(n;M)$-fold Darboux transformation\cite{t3e}, he obtained the higher order rogue wave solutions of Eq. (\ref{2}). Recently, as a generalization model of Eq. (\ref{2}), the GKMN equation has aroused people's more extensive interests and attentions.
In 2018, Peng applied ansatz method to obtain the bright soliton, dark soliton and power series solutions of Eq. (\ref{1}) and constructed the complexitons through the tanh method\cite{t4}. In 2019, Yildilm got bright, dark, singular, combo bright-dark, combo singular and singular periodic solitons of Eq. (\ref{1}) by using the modified simple equation method, Riccati function method and so on\cite{t4a,t4b,t4c,t4d,t4e}. In the same year, Ekici explored the plane wave solutions of Eq. (\ref{1}) via the extended trial function method\cite{t4ff}. Later, Kudryashov used the Jacobi elliptic functions\cite{t4f} to construct the general solution of the Eq. (\ref{1}).  With the aid of extended rational sinh-Gordon equation expansion method\cite{t5a}, Sulaiman obtained the trigonometric functions solutions of Eq. (\ref{1}). In addition, Jhangeer applied the direct extended algebraic approach to Eq. (\ref{1}) to derive the complex waves\cite{t5b}. In 2020, Rizvi\cite{t6a}  got dark, bright, periodic U-shaped and singular solitons through the generalized Kudryashov method. Subsequently, Kumar\cite{t6aa} discussed singular, dark, combined dark-singular solitons and other hyperbolic solutions by using the csch method, extended tanh-coth method and extended rational sinh-cosh method. Meanwhile, Talarposhti\cite{t7} and Ghanbari\cite{t7a} derived some new solitary solutions by using the Exp-function method. More recently, Rezazadeh constructed the analytical solutions of Eq. (\ref{1}) by utilizing the functional variable method\cite{t7b}.
\par Although various concise and efficient methods have been put forward to obtain  so many profound results about traveling wave solutions of the Eq. (\ref{1}), there still exist some problems unsolved. Firstly, due to the limitations caused by both the ansatz equations and the assumption about solutions in these direct methods, some solutions of Eq. (\ref{1}) could be lost. In addition, we note that unbounded traveling wave solutions of Eq. (\ref{1}) have not been reported in previous work. So, in this paper, we will try to apply the bifurcation method of dynamical system\cite{t8,t8a} to  solve these problems. This method allows detailed analysis of phase space geometry of the traveling wave system of Eq. (\ref{1}). Through exploring various orbits of traveling wave system of Eq. (\ref{1}), we construct its traveling wave solutions uniformly,  including the bounded traveling wave solutions and unbounded ones.
\par Our paper is organized as follows: In  section \ref{sec2}, we derive two types of traveling wave systems of Eq. (\ref{1}), including a singular traveling wave system. Then, by studying bifurcation of traveling wave solutions, we try to give global phase portraits of the two types of traveling wave systems. In section \ref{sec3}, by calculating some complicated elliptic integrals along  various orbits, we construct bounded and unbounded traveling wave solutions of Eq. (\ref{1}) uniformly.
\section{Traveling wave systems and bifurcation analysis}\label{sec2}
\par In this section, inspired by previous work \cite{t4a,t4f}, we firstly derive two types of traveling wave systems of Eq. (\ref{1}). Then, we study bifurcation of the two traveling wave systems (\ref{8}) and (\ref{13}) by dynamical system method.
\subsection{Two types of traveling wave systems of the GKMN equation}
\par Firstly, we assume that the type I traveling wave solution has the form
\begin{equation}\label{6}
q(x,y,t)=p(\xi){\rm exp}(\psi(x,y,t)i),
\end{equation}
where real function $p(\xi)$ is the portion of the amplitude with $\xi=x+my-ct$ and real function $\psi(x,y,t)=\kappa x+\omega y-rt+\theta$ is the phase component. Parameters
$m$,  $c$, $r$ and $\theta$ are reals and represent the inverse width of the soliton in the direction $y$, wave velocity, wave number and phase constant respectively. Real parameters
$\kappa$ and $\omega$ denote the frequencies along the $x$ and $y$ directions respectively.
\par Substituting the solution (\ref{6}) into Eq. (\ref{1}), we obtain
\begin{equation}\label{7}
\left\{
\begin{array}{lll}
  $real part$:amp^{''}+(r-a\kappa\omega)p+2\kappa bp^{3}=0,\\
 $imaginary part$:(-c+am\kappa+a\omega)p^{'}=0,
 \end{array}
 \right.
\end{equation}
where $'$ denotes $ d/d\xi $. From the imaginary part in (\ref{7}), we have
$$
   c=am\kappa+a\omega.
$$
It means that system (\ref{7}) has the equivalent form
\begin{equation}\label{8}
\left\{
\begin{array}{lll}
 p^{'}=y,\\
 y^{'}=\displaystyle\frac{a\kappa\omega-r}{am}p-\displaystyle\frac{2\kappa b}{am}p^{3},
\end{array}
\right.
\end{equation}
with constraint $c=am\kappa+a\omega$. Obviously, system (\ref{8}) is a Hamiltonian system with the energy function
\begin{equation}\label{9}
  H(p,y)=\displaystyle\frac{1}{2}y^{2}-\displaystyle\frac{a\kappa\omega-r}{2am}p^{2}+\displaystyle\frac{\kappa b}{2am}p^{4}.
\end{equation}
\par Then, we consider the type II traveling wave solution
\begin{equation}\label{10}
  q(x,y,t)=\phi(\xi){\rm exp}((\varphi (\xi)-\mu t)i),~~\xi=x+my-ct,
\end{equation}
 where $\phi(\xi)$ is the amplitude function, $\varphi(\xi)$  is the phase function and real parameters $m$, $c$ and $\mu$ represent the inverse width of the soliton in the direction $y$, wave velocity and wave frequency respectively. Here, we suppose that the amplitude function $\phi(\xi)\neq0$, otherwise the solution $q(x,y,t)$ of Eq. (\ref{1}) degenerates to the trivial solution. Substituting (\ref{10}) into (\ref{1}) and separating the real and imaginary parts, we have
\begin{equation}\label{11}
\left\{
\begin{array}{lll}
  $real part$: am\phi^{''}+n\phi\varphi^{'}+c\phi-am\phi(\varphi^{'})^{2}+2b\phi^{3}\varphi^{'}=0,\\
  $imaginary part$: -c\phi^{'}+2am\phi^{'}\varphi^{'}+am\phi\varphi^{''}=0,
 \end{array}
 \right.
\end{equation}
where $'$ denotes $d/d\xi$. From the second equation in (\ref{11}), we can solve
\begin{equation}\label{12}
   \varphi^{'}=\displaystyle\frac{e}{am\phi^{2}}+\displaystyle\frac{c}{2am},
\end{equation}
where $e$ is the integral constant. Plugging (\ref{12}) into the real part in (\ref{11}), we obtain
\begin{equation}\label{13}
\left\{
\begin{array}{lll}
 \phi^{'}=y,\\
 y^{'}=\displaystyle\frac{\alpha_{1}\phi^{6}+\alpha_{2}\phi^{4}+\alpha_{3}}{\phi^{3}},
\end{array}
\right.
\end{equation}
where $\alpha_{1}=-\displaystyle\frac{cb}{a^{2}m^{2}}$, $\alpha_{2}=-\displaystyle\frac{\mu}{am}-\displaystyle\frac{c^{2}+8be}{4}$ and $\alpha_{3}=\displaystyle\frac{e^{2}}{a^{2}m^{2}}\neq0$. If $\alpha_{3}=0$, system (\ref{13}) degenerates to system (\ref{8}).  System (\ref{13}) is a singular traveling wave system. With the transformation $d\xi=\phi^{3}d\zeta$, it can be converted to the associated regular system
\begin{equation}\label{14}
\left\{
\begin{array}{lll}
 \phi^{'}=\phi^{3}y,\\
 y^{'}=\alpha_{1}\phi^{6}+\alpha_{2}\phi^{4}+\alpha_{3},
\end{array}
\right.
\end{equation}
where $'$ stands for $d/d\zeta$, which has the energy function
\begin{equation}\label{15}
  H(\phi,y)=\displaystyle\frac{1}{2}y^{2}-\displaystyle\frac{\alpha_{1}}{4}\phi^{4}-\displaystyle\frac{\alpha_{2}}{2}\phi^{2}+\displaystyle\frac{\alpha_{3}}{2}\frac{1}{\phi^{2}}.
\end{equation}
When $\phi\neq0$, systems (\ref{13}) has the same vector field as  system (\ref{14}). So,  the function (\ref{15}) is also the energy function of system (\ref{13}).
\subsection{Bifurcation analysis}
\par Firstly, we study the distribution and properties of equilibria of system (\ref{8}) and (\ref{14}).
\begin{thm}\label{th1}
      When $\kappa b(a\kappa \omega-r)>0$, system (\ref{8}) has three equilibria $P_{1}(-\sqrt{\frac{a\kappa \omega-r}{2\kappa b}},0)$, $P_{2}(0,0)$ and $P_{3}(\sqrt{\frac{a\kappa \omega-r}{2\kappa b}},0)$. If $\displaystyle\frac{\kappa b}{am}<0(>0)$, $P_{1}$ and $P_{3}$ are saddles(centers), while $P_{2}$ is a center(saddle). When $\kappa b(a\kappa \omega-r)<0$, system (\ref{8}) has only one simple equilibrium $P_{2}(0,0)$. If $\displaystyle\frac{\kappa b}{am}<0(>0)$, $P_{2}$ is a saddle(center). When $a\kappa \omega-r=0$, system (\ref{8}) has a unique degenerate equilibrium $P_{2}(0,0)$. If $\displaystyle\frac{\kappa b}{am}<0(>0)$, $P_{2}$ is still a saddle(center).\\
\end{thm}
\begin{proof}
 When $\kappa b(a\kappa \omega-r)\neq0$, by the theory of dynamical system and the properties of Hamiltonian system\cite{t8,t8a,t8aa}, it is easy to verify the corresponding
 results above.
\par Especially, when $a\kappa \omega-r=0$, we see that system (\ref{8}) has a unique equilibrium $P_{2}(0,0)$ with the degenerate Jacobian matrix
$$
M(P_{2})= \left[
 \begin{matrix}
   0 &\qquad  1 \\
   0 &\qquad  0 \\
  \end{matrix}
  \right].\\
 $$
In this case, system (\ref{8}) has the associated normal form
\begin{equation*}
\left\{
\begin{array}{lll}
 p^{'}=y,\\
 y^{'}=a_{k}p^{k}[1+f_{1}(p)]+b_{n}p^{n}y[1+f_{2}(p)]+y^{2}f_{3}(p,y)=-\displaystyle\frac{2\kappa b}{am}p^{3},
\end{array}
\right.
\end{equation*}
where $k=3$, $a_{k}=\displaystyle-\frac{2\kappa b}{am}$, $b_{n}=0$, $f_{1}(p)=0$, $f_{2}(p)=0$ and $f_{3}(p,y)=0$. From the fact that $k$ is an odd number, the degenerate equilibrium $P_{2}$ is a saddle when $a_{k}>0$, whereas $P_{2}$ is a center when $a_{k}<0$ and $b_{n}=0$ according to the qualitative theory of differential equation\cite[Theorem 7.2, Chapter 2]{t8aa}.
\end{proof}
Furthermore, we note that when $\kappa b(a\kappa \omega-r)>0$, the energy of three equilibria has the following relationship
$$
\begin{aligned}
h(P_{2})&=0,\\
h(P_{1})&\equiv h(P_{3})=\displaystyle-\frac{(a\kappa \omega-r)^2}{8ab\kappa m},\\
\end{aligned}
$$
which means that the energy of saddles $P_{1}$ and $P_{3}$ is always equivalent. So, according to the properities of Hamilton system \cite{t8}, we have the following
results:
\par\noindent{\bf Case 1.} When $\displaystyle \kappa b(a\kappa \omega-r)>0$ and $\displaystyle\frac{\kappa b}{am}<0$, there exist two heteroclinic orbits $\Gamma^{0}$ and
$\Gamma_{0}$ connecting the saddles $P_{1}$ and $P_{3}$. Center $P_{2}$ is surrounded by a family of periodic orbits
\begin{equation*}
\gamma(h)=\{H(p,y)=h,h\in\big(0,\displaystyle-\frac{(a\kappa \omega-r)^2}{8ab\kappa m}\big)\}.
\end{equation*}
$\gamma(h)$ tends to $P_{2}$ as h$\rightarrow$$0$, and tends to $\Gamma^{0}$ and $\Gamma_{0}$ as h$\rightarrow$$\displaystyle-\frac{(a\kappa \omega-r)^2}{8ab\kappa m}$. Except the periodic orbit and the heteroclinic orbits, other orbits of system (\ref{8}) are unbounded. Please see Fig. \ref{fig1a}.
\par\noindent {\bf Case 2.}  When $\displaystyle\kappa b(a\kappa \omega-r)>0$ and $\displaystyle\frac{\kappa b}{am}>0$, all orbits are bounded. There exist two homoclinic orbits $\Upsilon_{1}$ and
$\Upsilon_{2}$ connecting the saddle $P_{2}$. Centers $P_{1}$ and $P_{3}$ are surrounded by two families of periodic orbits
$$
\gamma_{_R}(h)=\{H(p,y)=h,h\in\big(\displaystyle-\frac{(a\kappa \omega-r)^2}{8ab\kappa m},0\big)\},
$$
$$
\gamma_{_L}(h)=\{H(p,y)=h,h\in\big(\displaystyle-\frac{(a\kappa \omega-r)^2}{8ab\kappa m},0\big)\}.
$$
$\gamma_{_R}(h)$ and $\gamma_{_L}(h)$ tend to $P_{1}$ and $P_{3}$ respectively as h$\rightarrow$$\displaystyle-\frac{(a\kappa \omega-r)^2}{8ab\kappa m}$, and tend to
$\Upsilon_{1}$ and $\Upsilon_{2}$ respectively as h$\rightarrow$$0$ shown in Fig. \ref{fig1b}.
\par\noindent {\bf Case 3.} When $\displaystyle\kappa b(a\kappa \omega-r)<0$ or $a\kappa \omega-r=0$, system (\ref{8}) has only one equilibrium. If $\displaystyle\frac{\kappa b}{am}<0(>0)$, all orbits of system (\ref{8}) are unbounded(bounded) shown in Figs. \ref{fig2} and \ref{fig3}.\\
\begin{figure*}[!h]
\centering
\centering
\subfigure[$\displaystyle\frac{\kappa b}{am}=-\frac{1}{2}$, $\displaystyle \frac{a\kappa \omega-r}{am}=-4$.]{\label{fig1a}\includegraphics[height=5.8cm,width=5.8cm]{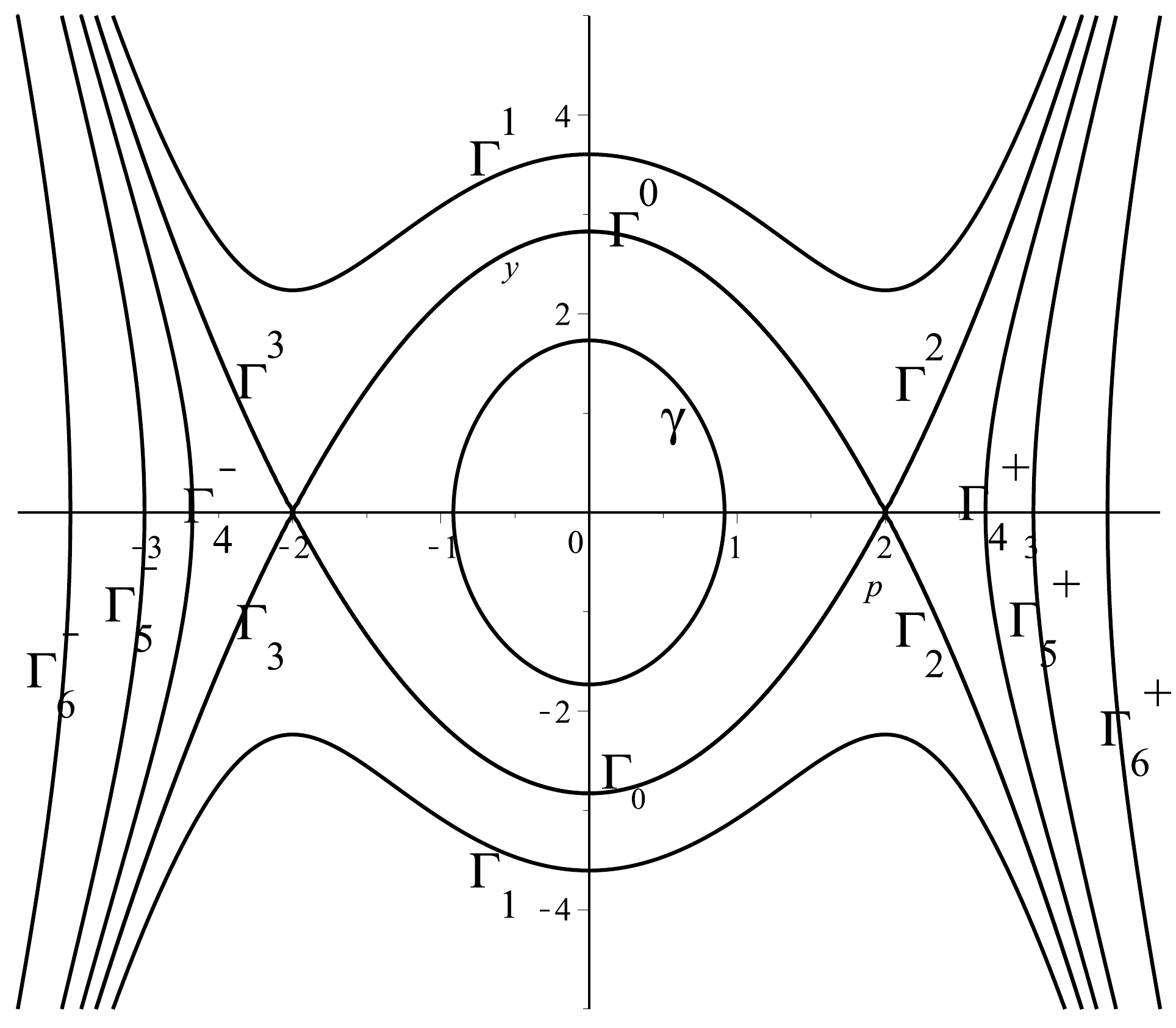}}
\subfigure[$\displaystyle\frac{\kappa b}{am}=\frac{1}{2}$, $\displaystyle\frac{a\kappa \omega-r}{am}=4$.]{\label{fig1b}\includegraphics[height=5.8cm,width=5.8cm]{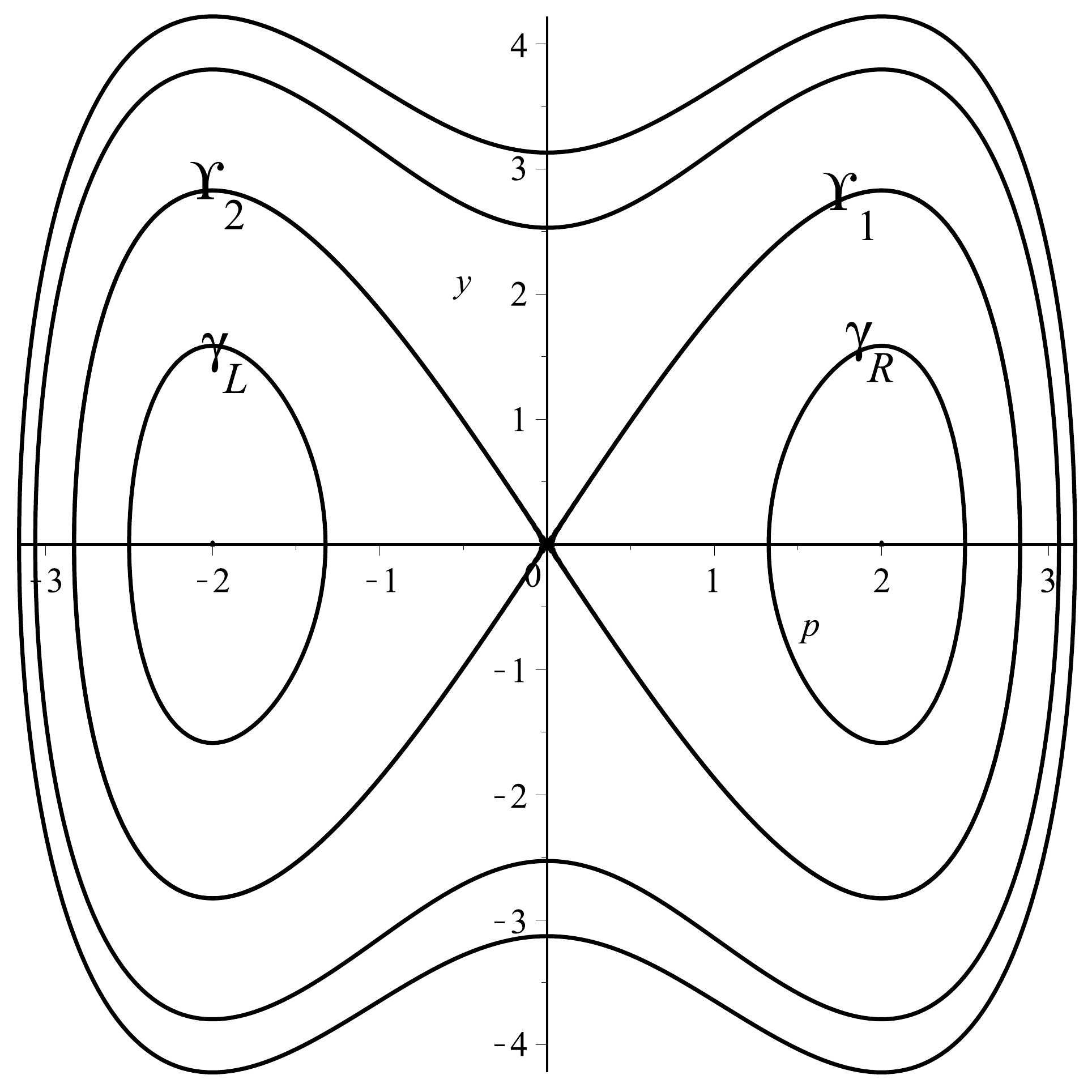}}
\caption{Phase portraits of system (\ref{8}) for $\kappa b(a\kappa \omega-r)>0$.}
\label{fig1}
\end{figure*}
\begin{figure*}[!h]
\centering
\subfigure[$\displaystyle \frac{\kappa b}{am}=-2$, $\displaystyle \frac{a\kappa \omega-r}{am}=4$.]{\label{fig2a}\includegraphics[height=5.8cm,width=5.8cm]{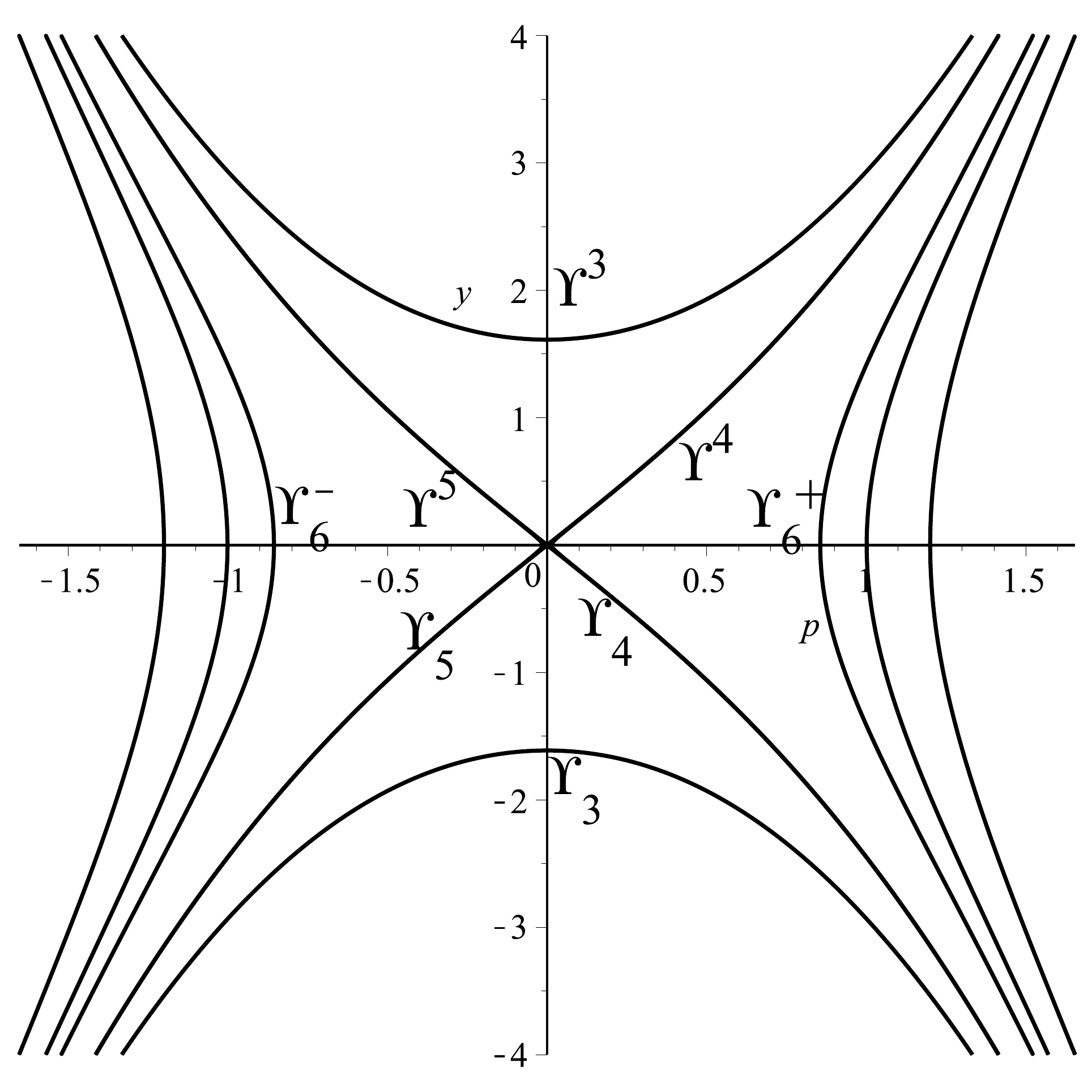}}
\subfigure[$\displaystyle \frac{\kappa b}{am}=2$, $\displaystyle \frac{a\kappa \omega-r}{am}=-4$.]{\label{fig2b}\includegraphics[height=5.8cm,width=5.8cm]{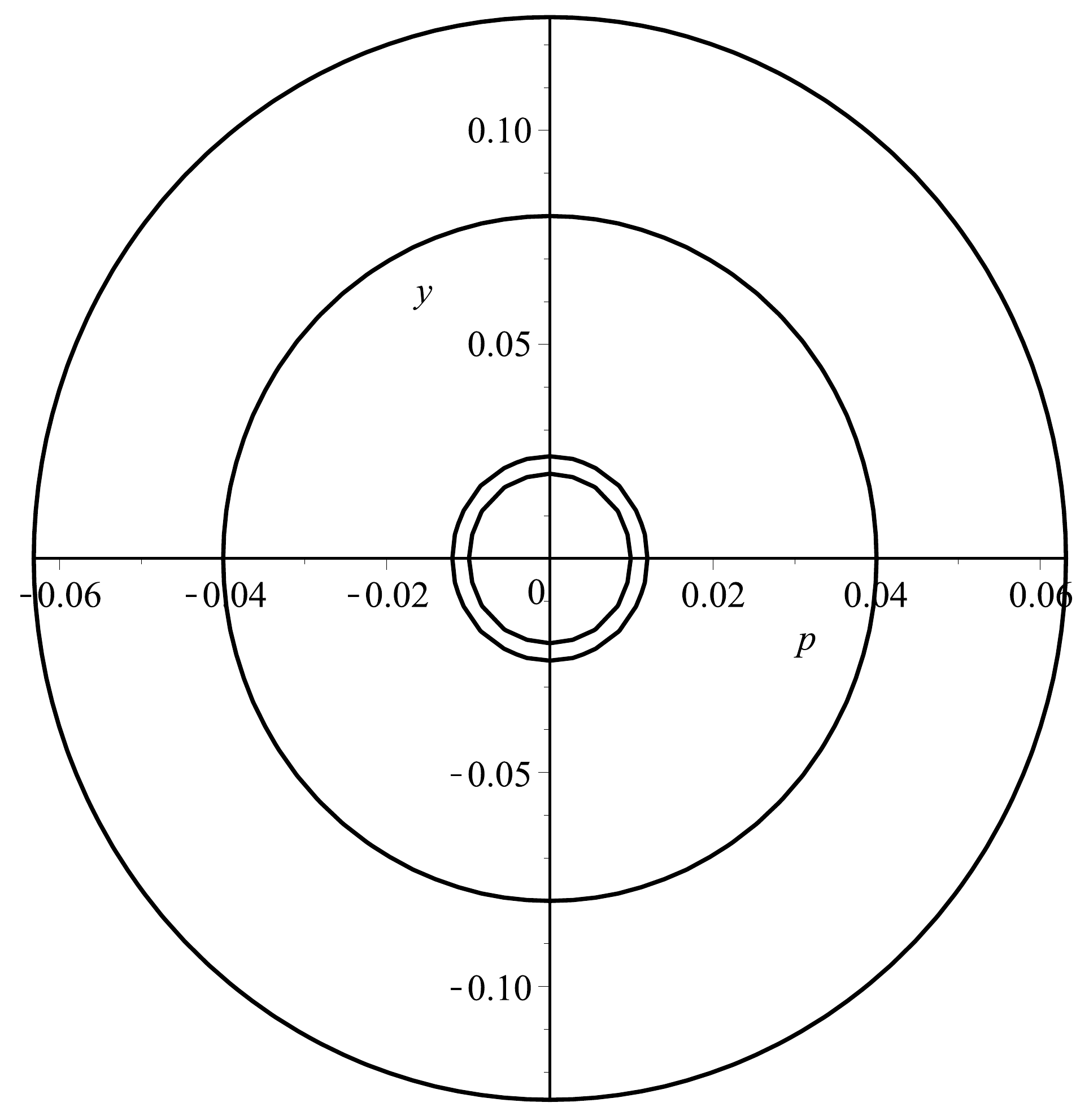}}
\caption{Phase portraits of system (\ref{8}) for $\kappa b(a\kappa \omega-r)<0$.}
\label{fig2}
\end{figure*}
\begin{figure*}[!h]
\centering
\subfigure[$\displaystyle \frac{\kappa b}{am}=-\frac{1}{2}$, $a\kappa \omega-r=0$.]{\label{fig3a}\includegraphics[height=5.8cm,width=5.8cm]{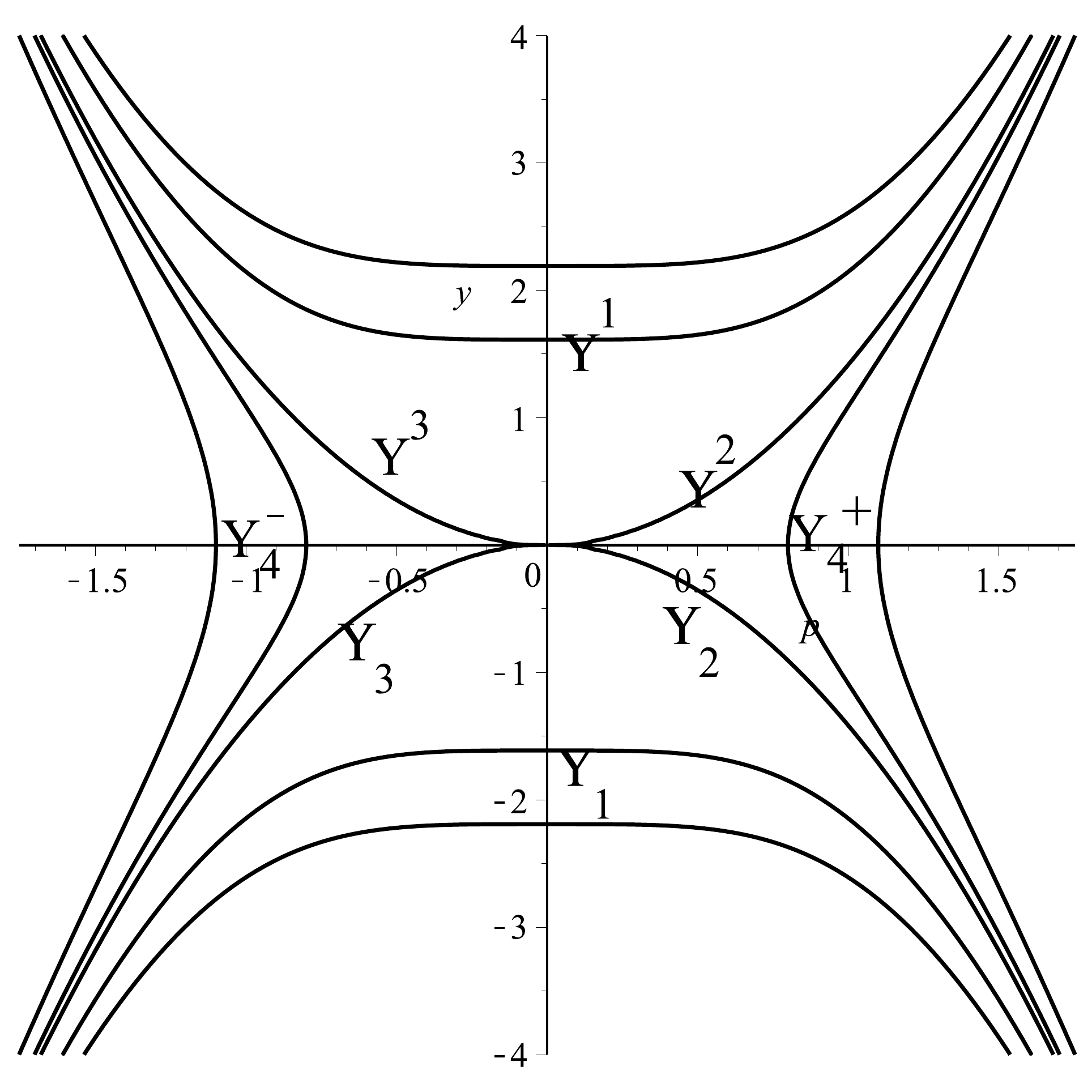}}
\subfigure[$\displaystyle \frac{\kappa b}{am}=\frac{1}{2}$, $a\kappa \omega-r=0$.]{\label{fig3b}\includegraphics[height=5.8cm,width=5.8cm]{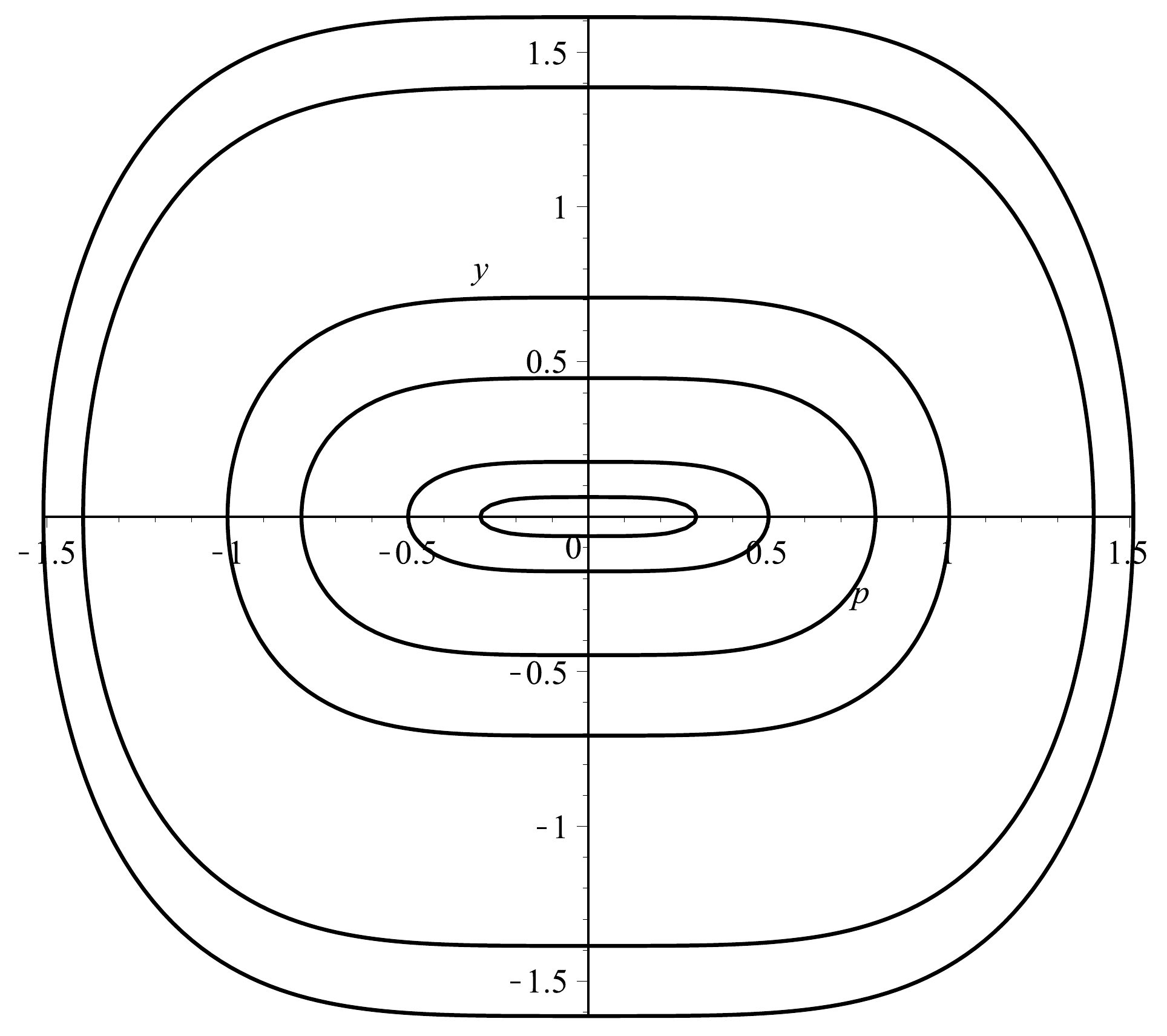}}
\caption{Phase portraits of system (\ref{8}) for $a\kappa \omega-r=0$.}
\label{fig3}
\end{figure*}
\begin{thm}\label{th2}
      When $\alpha_{1}>0$, $\alpha_{2}<0$ and $0<\alpha_{3}<\displaystyle-\frac{4\alpha_{2}^{3}}{27\alpha_{1}^{2}}$, system (\ref{14}) has two pairs of equilibria
      $\widetilde{P}_{1,2}(\pm u_{1},0)$ and $\widetilde{P}_{3,4}(\pm u_{2},0)$, where $\widetilde{P}_{1,2}$ are centers and $\widetilde{P}_{3,4}$ are saddles.
      When $\alpha_{1}>0$, $\alpha_{2}<0$ and $\alpha_{3}=\displaystyle-\frac{4\alpha_{2}^{3}}{27\alpha_{1}^{2}}$, system (\ref{14}) has a pair of cusps
     $\widetilde{P}_{5,6}(\pm u_{3},0)$.
      When either $\alpha_{1}>0$, $\alpha_{2} \geq 0$ or $\alpha_{1}>0$, $\alpha_{2}<0$, $\alpha_{3}>\displaystyle-\frac{4\alpha_{2}^{3}}{27\alpha_{1}^{2}}$, system (\ref{14}) has no equilibrium.
      When $\alpha_{1}<0$, system (\ref{14}) has a pair of centers $\widetilde{P}_{7,8}(\pm u_{4},0)$.
     \end{thm}
\begin{proof}
 Similar to the proof of Theorem {\ref{th1}}, we only give the detail proof for degenerate equilibria for simplicity. When $\alpha_{1}>0$, $\alpha_{2}<0$ and $\alpha_{3}=\displaystyle -\frac{4\alpha_{2}^{3}}{27\alpha_{1}^{2}}$, a direct computation shows that system (\ref{14}) has a pair of degenerate equilibria $\widetilde{P}_{5,6}(\pm{u_{3}},0)$ with the degenerate Jacobian matrix
 $$
M(\widetilde{P}_{5,6})= \left[
 \begin{matrix}
   0 &\qquad  \pm u_{3}^{3} \\
   0 &\qquad  0 \\
  \end{matrix}
  \right].\\
 $$
In order to judge the type of the equilibrium $\widetilde{P}_{5}({u_{3}},0)$, we make the homeomorphic transformation
\begin{equation}\label{16}
u=\phi-u_{3},v=y+\displaystyle\frac{1}{u^{3}_{3}}(3u^{2}_{3}\alpha y+3u_{3}{\alpha}^{2}y+\alpha^{3}y),
\end{equation}
which transforms system (\ref{14}) to the normal form
\begin{equation*}
\left\{
\begin{array}{lll}
 u^{'}=v,\\
 v^{'}=a_{k}u^{k}[1+g_{1}(u)]+b_{n}u^{n}v[1+g_{2}(u)]+v^{2}g_{3}(u,v),
\end{array}
\right.
\end{equation*}
where $k=2$, $M=\displaystyle\frac{\alpha_{2}+3u_{3}^{2}\alpha_{1}}{\alpha_{1}}$, $a_{k}=\displaystyle\frac{4\alpha_{1}M}{u_{3}}$,
$g_{1}(u)=\displaystyle\frac{1}{4u_{3}^{5}M}((8u_{3}^{5}M+16u_{3}^{4}M)u+(12u_{3}^{5}+24u_{3}^{4}+25u_{3}^{3}M)u^{2}
+(42u_{3}^{4}+24u_{3}^{3}+19u_{3}^{2}M)u^{3}+(54u_{3}^{3}+9u_{3}^{2}+7u_{3}M)u^{4}+(33u_{3}^{2}+M)u^{5}+9u_{3}u^{6}+u^{7})$, $b_{n}=0$ and
$g_{3}(u,v)=\displaystyle\frac{3u_{3}^{2}+6u_{3}u+3u^{2}}{u_{3}^{3}+3u_{3}^{2}u+3u_{3}u^{2}+u^{3}}$. From the fact that $k=2$ is an even number and $b_{n}=0$, we come to the conclusion that equilibrium $\widetilde{P}_{5}({u_{3}},0)$ is a cusp according to the qualitative theory of differential equation \cite[Theorem 7.3, Chapter 2]{t8aa}.
\par Similarly, applying another homeomorphic transformation
\begin{equation}\label{17}
u=\phi+u_{3},v=y-\displaystyle\frac{1}{u^{3}_{3}}(3u^{2}_{3}\alpha y-3u_{3}{\alpha}^{2}y+\alpha^{3}y),
 \end{equation}
to system (\ref{14}), one can check equilibrium $\widetilde{P}_{6}({-u_{3}},0)$ is also a cusp.
\end{proof}
\par  Based on the properties of the system (\ref{14}), the bifurcation results of system (\ref{13}) are given as follows.
\par\noindent {\bf Case I.} When $\alpha_{1}>0$, $\alpha_{2}<0$ and $0<\alpha_{3}<\displaystyle-\frac{4\alpha_{2}^{3}}{27\alpha_{1}^{2}}$, there exist two homoclinic orbits
$\Pi^{^{+}}_{0}$ and $\Pi_{0}^{^{-}}$ connecting saddles $\widetilde{P}_{3}$ and $\widetilde{P}_{4}$ respectively. Centers $\widetilde{P}_{1}$ and $\widetilde{P}_{2}$ are surrounded by two families of periodic orbits, respectively
$$
\begin{aligned}
&\gamma^{^{+}}_{_1}(h)=\{H(\phi,y)=h,h\in\big(h(u_{1},0),h(u_{2},0)\big)\},\\
&\gamma^{^-}_{_1}(h)=\{H(\phi,y)=h,h\in\big(h(-u_{1},0),h(-u_{2},0)\big)\}.\\
\end{aligned}
$$
$\gamma^{^{+}}_{_1}(h)$ and $\gamma^{^-}_{_1}(h)$ respectively tend to $\widetilde{P}_{1,2}(\pm u_{1},0)$ as h$\rightarrow$$h(\pm u_{1},0)$, and tend to $\Pi^{^{+}}_{0}$ and $\Pi_{0}^{^{-}}$ as
h$\rightarrow$$h(\pm u_{2},0)$. All orbits are unbounded except for the periodic orbits and the homoclinic orbits shown in Fig. \ref{fig4a}.
\par\noindent {\bf Case II.}  When $\alpha_{1}>0$, $\alpha_{2}<0$ and $\alpha_{3}=\displaystyle-\frac{4\alpha_{2}^{3}}{27\alpha_{1}^{2}}$, all orbits of system (\ref{13}) are
unbounded. In the positive $\phi$-axis, the orbit $\Omega^{1}$ is different from orbit $\Omega_{1}$. More precisely, the $\omega$-limit set of $\Omega^{1}$ and the $\alpha$-limit
set of $\Omega_{1}$ correspond to the same equilibrium $\widetilde{P}_{5}(u_{3},0)$  shown in Fig. \ref{fig4b}.
\par\noindent {\bf Case III.} When $\alpha_{1}>0$, $\alpha_{2}\geq0$ or $\alpha_{1}>0$, $\alpha_{2}<0$ and $\alpha_{3}>\displaystyle-\frac{4\alpha_{2}^{3}}{27\alpha_{1}^{2}}$, there only exist unbounded orbits of system (\ref{13}) shown in Fig. \ref{fig4c}.
\par\noindent {\bf Case IV.} When $\alpha_{1}<0$, there exist two families of periodic orbits shown in Fig. \ref{fig4d}.
\begin{figure*}[!h]
\centering
\subfigure[$\alpha_{1}=1$,$\alpha_{2}=-4$,$\alpha_{3}=0.1$.]{\label{fig4a}\includegraphics[height=5.8cm,width=5.8cm]{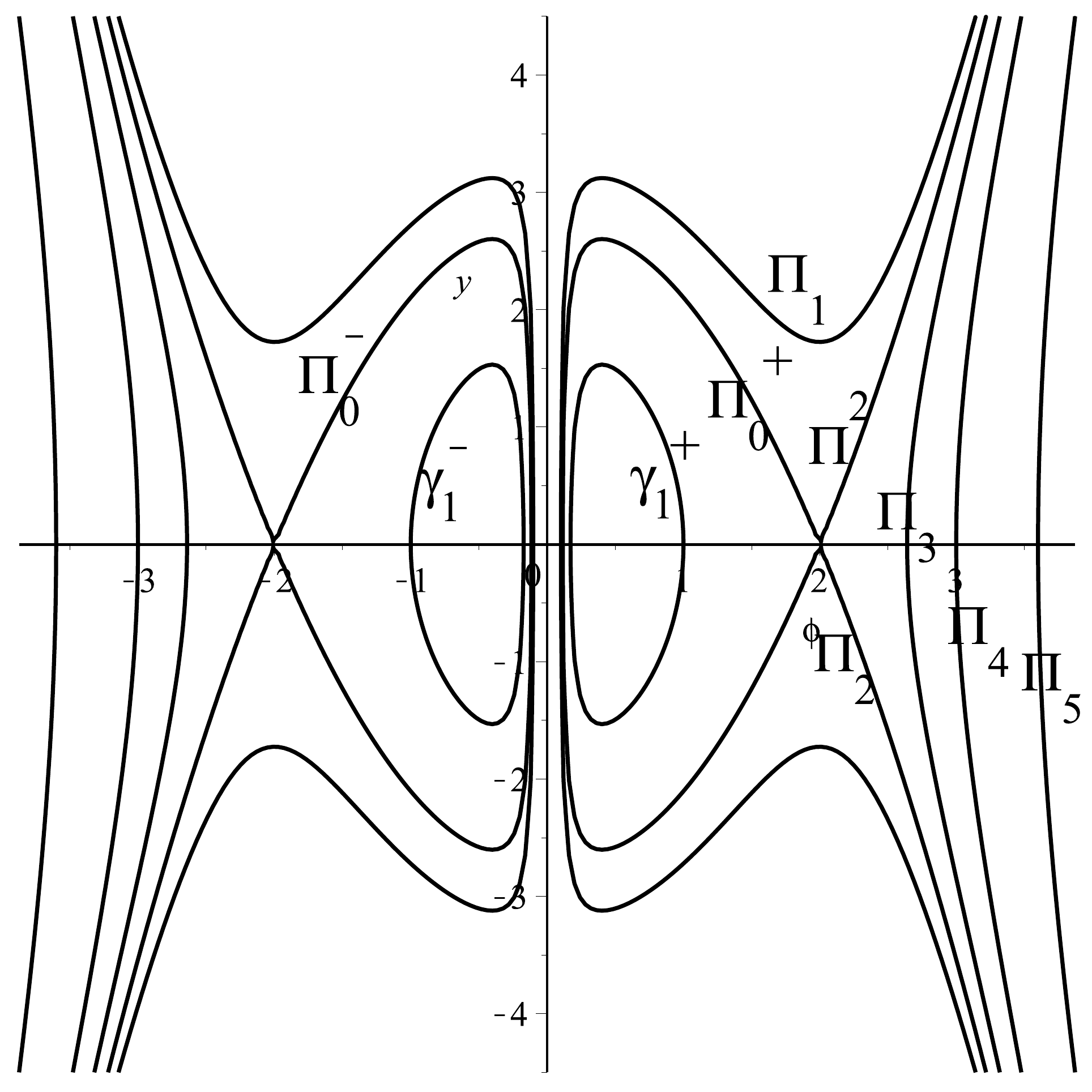}}
\subfigure[$\alpha_{1}=1$,$\alpha_{2}=-4$,$\alpha_{3}=\frac{256}{27}$.]{\label{fig4b}\includegraphics[height=5.8cm,width=5.8cm]{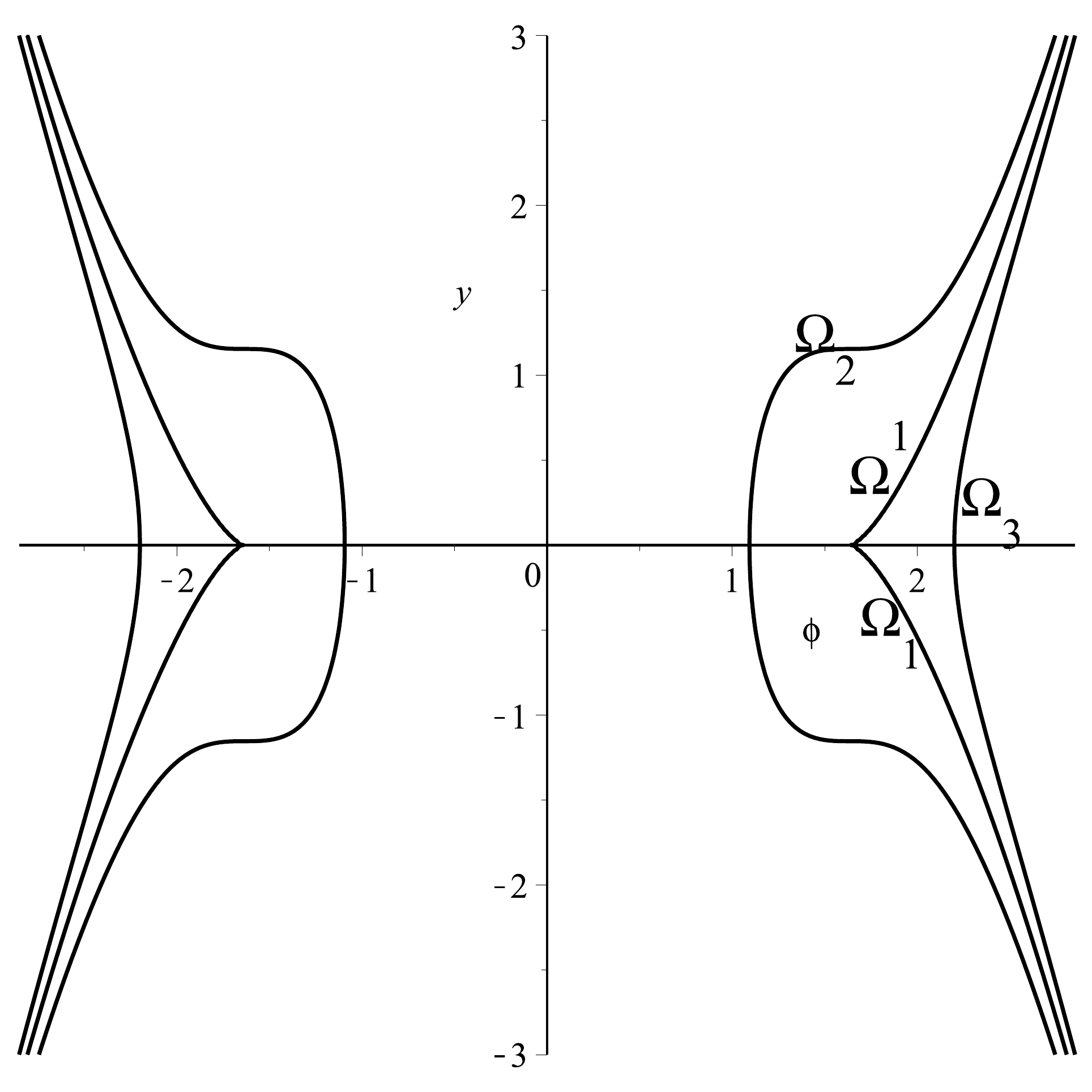}}
\subfigure[$\alpha_{1}=1$,$\alpha_{2}=0$,$\alpha_{3}=0.1$.]{\label{fig4c}\includegraphics[height=5.8cm,width=5.8cm]{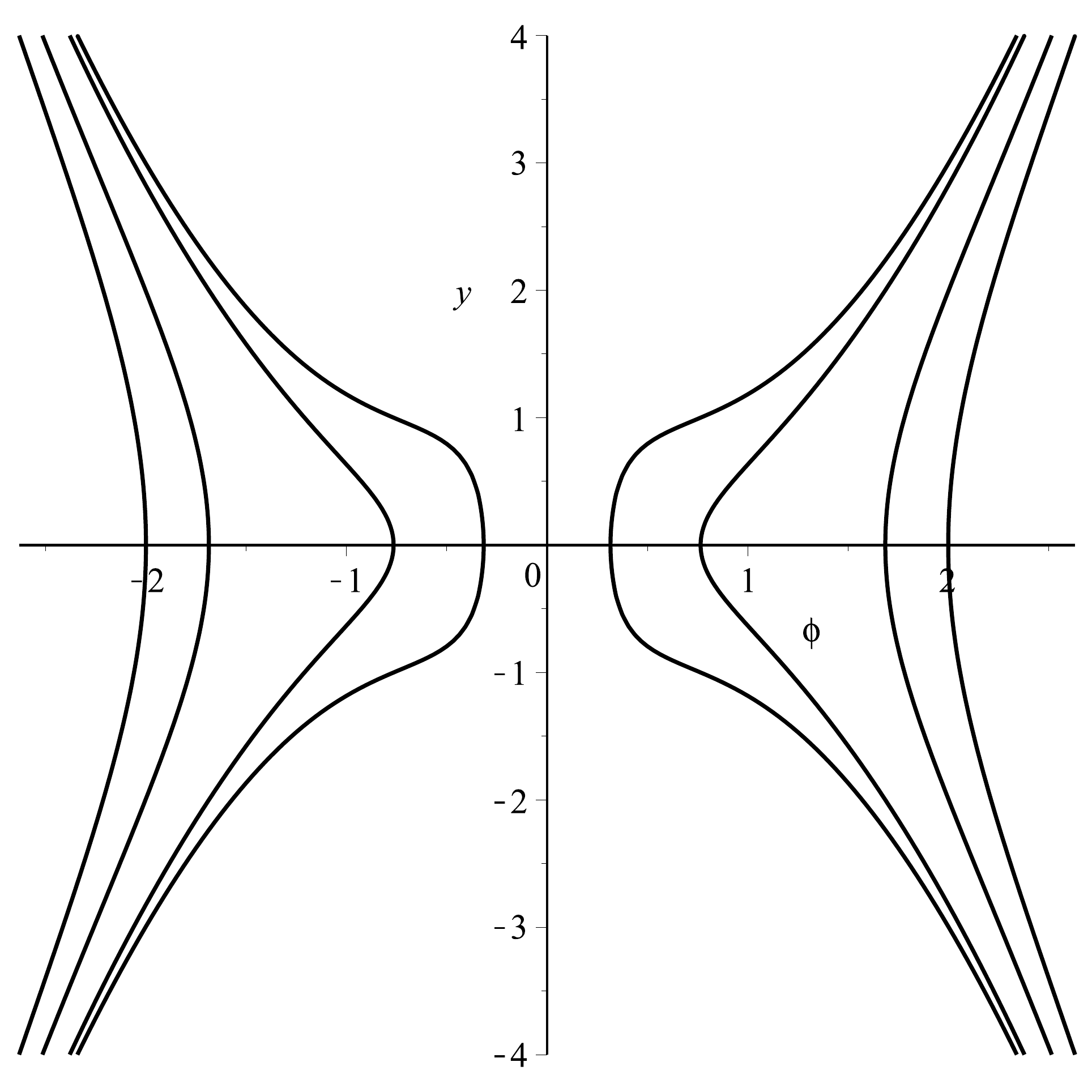}}
\subfigure[$\alpha_{1}=-1$,$\alpha_{2}=0$,$\alpha_{3}=0.1$.]{\label{fig4d}\includegraphics[height=5.8cm,width=5.8cm]{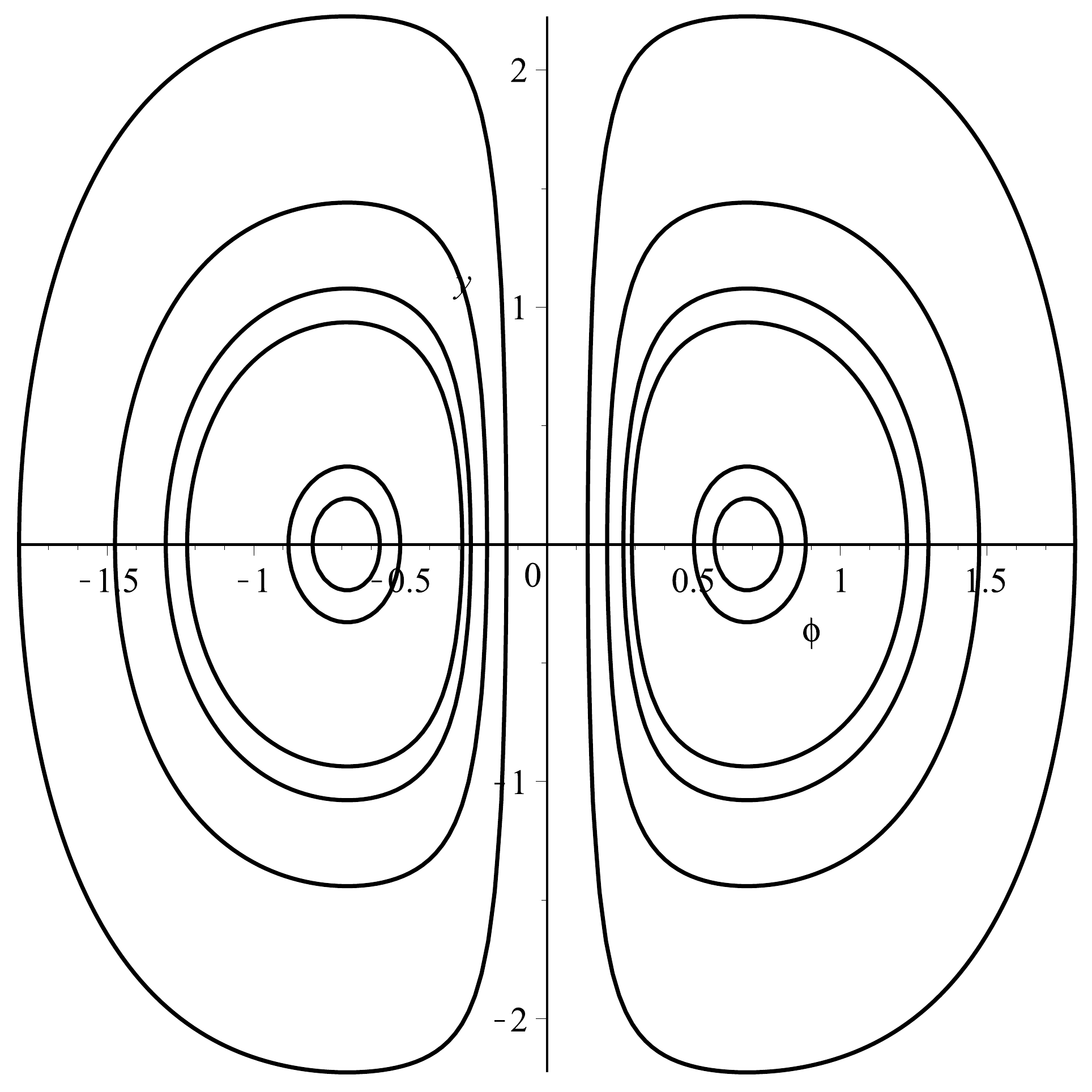}}
\caption{Phase portraits of system (\ref{13}).}
\label{fig4}
\end{figure*}
\section{Exact solutions of system (\ref{8}) and (\ref{13})}\label{sec3}
\par In this section, we seek the explicit expressions of bounded and unbounded solutions of systems (\ref{8}) and (\ref{13}).
\subsection{Bounded solutions of system (\ref{8})}
\par According to the bifurcation results in Theorem \ref{th1}, to seek  bounded solutions of system (\ref{8}), there are three cases need to be discussed.
\par\textbf{1.} When $\kappa b(a\kappa \omega-r)>0$ and $\displaystyle\frac{\kappa b}{am}<0$, we  consider two subcases as follows.
\par\noindent(1) Consider the periodic orbits shown in Fig. {\ref{fig1a}}, whose energy is lower than energy of the saddle $P_{1}$, but higher than energy of center $P_{2}$. Any one of them can be expressed by
\begin{equation*}
  y=\pm\sqrt{-\displaystyle\frac{\kappa b}{am}}\sqrt{(p-p_{1})(p-p_{2})(p_{3}-p)(p_{4}-p)},
\end{equation*}
where $p_{1}$, $p_{2}$ , $p_{3}$ and $p_{4}$ satisfy the constraint condition $p_{1}<-\sqrt{\frac{a\kappa \omega-r}{2\kappa b}}<p_{2}<p<p_{3}<\sqrt{\frac{a\kappa \omega-r}{2\kappa b}}<p_{4}$. Assuming that the period is $2T_{0}$ and choosing
initial
value $p(0)=p_{2}$, we have
$$\int_{p_{2}}^{p}\sqrt{-\frac{am}{\kappa b}}\frac{dp}{\sqrt{(p-p_{1})(p-p_{2})(p_{3}-p)(p_{4}-p)}}= \int_{0}^{\xi}d\xi, ~~0<\xi<T_{0},$$
$$-\int_{p}^{p_{2}}\sqrt{-\frac{am}{\kappa b}}\frac{dp}{\sqrt{(p-p_{1})(p-p_{2})(p_{3}-p)(p_{4}-p)}}= \int_{\xi}^{0}d\xi, ~~ -T_{0}<\xi<0,$$
which can be rewritten as
$$\int_{p_{2}}^{p}\sqrt{-\frac{am}{\kappa b}}\frac{dp}{\sqrt{(p-p_{1})(p-p_{2})(p_{3}-p)(p_{4}-p)}}= \mid\xi\mid, ~~-T_{0}<\xi<T_{0}.$$
By calculating the elliptic integral
$$
  \int_{p_{2}}^{p}\frac{dp}{\sqrt{(p-p_{1})(p-p_{2})(p_{3}-p)(p_{4}-p)}}=g\cdot sn^{-1}(\sqrt{\frac{(p_{3}-p_{1})(p-p_{2})}{(p_{3}-p_{2})(p-p_{1})}},k),
$$
where $g=\displaystyle\frac{2}{\sqrt{(p_{3}-p_{1})(p_{4}-p_{2})}}$, $k^{2}=\displaystyle\frac{(p_{3}-p_{2})(p_{4}-p_{1})}{(p_{4}-p_{2})(p_{3}-p_{1})}$, we get the first type
of periodic solution of system (\ref{8})
\begin{equation*}
  p_{b_{1}}(\xi)=p_{1}+\frac{(p_{2}-p_{1})(p_{3}-p_{1})}{(p_{3}-p_{1})-(p_{3}-p_{2})sn^{2}(\sqrt{\frac{-\kappa b(p_{4}-p_{2})(p_{3}-p_{1})}{4am}}\xi)}, ~~-T_{0}<\xi<T_{0}.
\end{equation*}
\par\noindent(2) Consider the heteroclinic orbit $\Gamma^{0}$ shown in Fig. {\ref{fig1a}}, whose energy is equal to energy of the saddle $P_{1}$. It can be expressed by
\begin{equation*}
  y=\sqrt{-\displaystyle\frac{\kappa b}{am}}\sqrt{(p-p_{5})^{2}(p_{6}-p)^{2}},
\end{equation*}
where $p_{5}$ and $p_{6}$ satisfy the constraint condition $-\sqrt{\displaystyle \frac{a\kappa \omega-r}{2\kappa
b}}=p_{5}<p<p_{6}=\sqrt{\displaystyle\frac{a\kappa \omega-r}{2\kappa b}}$. Choosing initial
value $p(0)=\displaystyle\frac{p_{5}+p_{6}}{2}=0$, we have
$$\int_{0}^{p}\sqrt{-\frac{am}{\kappa b}}\frac{dp}{{(p-p_{5})(p_{6}-p)}}= \int_{0}^{\xi}d\xi, ~~-\infty<\xi<+\infty,$$
Noting that
$$
  \int_{0}^{p}\frac{dp}{{{(p-p_{5})(p_{6}-p)}}}=\frac{2}{p_{6}-p_{5}}\tanh^{-1}\frac{2p-(p_{5}+p_{6})}{p_{6}-p_{5}},
$$
we obtain the expression of kink wave solution of system (\ref{8})
$$
 p_{b_{2}}(\xi)=\frac{p_{6}-p_{5}}{2}\tanh(\frac{p_{6}-p_{5}}{2}\sqrt{-\frac{\kappa b}{am}}\xi),~~ -\infty<\xi<+\infty.
$$
\noindent Applying similar calculation to  another heteroclinic orbit $\Gamma_{0}$ shown in Fig. {\ref{fig1a}}, we can get the corresponding  kink wave solution of system (\ref{8}) as follows
$$
 p_{b_{2^{'}}}(\xi)=\frac{p_{6}-p_{5}}{2}\tanh(\frac{p_{5}-p_{6}}{2}\sqrt{-\frac{\kappa b}{am}}\xi),~~ -\infty<\xi<+\infty.
$$
\par\textbf{2.} When $\kappa b(a\kappa \omega-r)>0$ and $\displaystyle\frac{\kappa b}{am}>0$, we need to consider three subcases in this case.
\par\noindent(1) Consider the family of periodic orbits inside the homoclinic orbit  $\Upsilon_{2}$ shown in Fig. {\ref{fig1b}}, whose energy is lower than energy of the saddle $P_{2}$, but higher than energy of center $P_{1}$. Any one of them can be expressed by
\begin{equation*}
  y=\pm\sqrt{\displaystyle\frac{\kappa b}{am}}\sqrt{(p-p_{7})(p_{8}-p)(p_{9}-p)(p_{10}-p)},
\end{equation*}
where $p_{7}$, $p_{8}$ , $p_{9}$ and $p_{10}$ satisfy the constraint condition $p_{7}<p<p_{8}<0<p_{9}<\sqrt{\frac{a\kappa \omega-r}{2\kappa b}}<p_{10}$. Assuming that the period is $2T_{1}$, similar to the calculation of solution $p_{b_{1}}(\xi)$, we get the second type of periodic solution of system (\ref{8})
\begin{equation*}
  p_{b_{3}}(\xi)=p_{10}-\frac{(p_{10}-p_{8})(p_{10}-p_{7})}{(p_{10}-p_{8})+(p_{8}-p_{7})sn^{2}(\sqrt{\frac{\kappa b}{am}}\frac{\sqrt{(p_{10}-p_{8})(p_{9}-p_{7})}}{2}\xi)},
  ~~-T_{1}<\xi<T_{1}.
\end{equation*}

\noindent Similarly, we can give  the periodic solutions of system (\ref{8}) corresponding to the family of periodic orbits inside the homoclinic orbit  $\Upsilon_{1}$ shown in Fig. {\ref{fig1b}}
\begin{equation*}
  p_{b_{3^{'}}}(\xi)=p_{8}+\frac{(p_{9}-p_{8})(p_{10}-p_{8})}{(p_{10}-p_{8})-(p_{10}-p_{9})sn^{2}(\sqrt{\frac{\kappa b}{am}}\frac{\sqrt{(p_{10}-p_{8})(p_{9}-p_{7})}}{2}\xi)},
  ~~-T_{1}<\xi<T_{1},
\end{equation*}
where $p_{7}<-\sqrt{\frac{a\kappa \omega-r}{2\kappa b}}<p_{8}<0<p_{9}<p<p_{10}$.

\par\noindent(2) Consider the homoclinic orbits shown in Fig. {\ref{fig1b}}, whose energy is equal to energy of the saddle $P_{2}$. The homoclinic orbit $\Upsilon_{2}$ can be expressed by
\begin{equation*}
  y=\pm\sqrt{\displaystyle\frac{\kappa b}{am}}\sqrt{(p+p_{11})p^{2}(p_{11}-p)},
\end{equation*}
where $p_{11}$ and $-p_{11}$  satisfy the constraint condition $-p_{11}<p<0<\sqrt{\frac{a\kappa \omega-r}{2\kappa b}}<p_{11}$. Choosing initial value $p(0)=-p_{11}$, we
have
$$\int_{-p_{11}}^{p}\sqrt{\frac{am}{\kappa b}}\frac{dp}{\sqrt{(p+p_{11})p^{2}(p_{11}-p)}}= \int_{0}^{\xi}d\xi, ~~\xi>0,$$
$$-\int_{p}^{-p_{11}}\sqrt{\frac{am}{\kappa b}}\frac{dp}{\sqrt{(p+p_{11})p^{2}(p_{11}-p)}}= \int_{\xi}^{0}d\xi, ~~ \xi<0,$$
 which can be rewritten as
 $$
 \int_{-p_{11}}^{p}\sqrt{\frac{am}{\kappa b}}\frac{dp}{\sqrt{(p+p_{11})p^{2}(p_{11}-p)}}=\mid\xi\mid.
 $$
Noting that
$$
  \int_{-p_{11}}^{p}\frac{dp}{p\sqrt{(p+p_{11})(p_{11}-p)}}= \frac{\ln(\frac{-p_{11}-\sqrt{{p_{11}}^{2}-p^{2}}}{p})}{-p_{11}},
$$
we obtain the expression of solitary wave solution of system (\ref{8})
$$
p_{b_{4}}(\xi)=\frac{-2p_{11}{\rm exp}(\sqrt{\frac{\kappa b}{am}}p_{11}\mid\xi\mid)}{{{\rm exp}(2\sqrt{\frac{\kappa b}{am}}p_{11}\mid\xi\mid)}+1},~~-\infty<\xi<+\infty.
$$
Similarly, we can get another solitary wave solution of system (\ref{8}) corresponding to the homoclinic orbit $\Upsilon_{1}$
$$ p_{b_{4^{'}}}(\xi)=\frac{2p_{11}{\rm exp}(\sqrt{\frac{\kappa b}{am}}p_{11}\mid\xi\mid)}{{{\rm exp}(2\sqrt{\frac{\kappa b}{am}}p_{11}\mid\xi\mid)}+1},~~-\infty<\xi<+\infty,$$
where $-p_{11}<-\sqrt{\frac{a\kappa \omega-r}{2\kappa b}}<0<p<p_{11}$.

\par\noindent(3) Consider the family of large amplitude periodic orbits shown in Fig. {\ref{fig1b}}, whose energy is higher than energy of the saddle $P_{2}$. Any one of them can be expressed by
\begin{equation*}
  y=\pm\sqrt{\displaystyle\frac{\kappa b}{am}}\sqrt{(p+p_{12})(p_{12}-p)(p^{2}+p_{12}^{2}-\displaystyle\frac{a\kappa \omega-r}{\kappa b})},
\end{equation*}
where $p_{12}$ and $-p_{12}$ satisfy the constraint condition $p_{12}>0$, $p_{12}^{2}-\displaystyle\frac{a\kappa \omega-r}{\kappa b}>0$ and
$-p_{12}<p<p_{12}$. Assuming that the period is $2T_{2}$ and choosing initial value $p(0)=\displaystyle\frac{-p_{12}+p_{12}}{2}=0$, we
have
$$\int_{0}^{p}\sqrt{\displaystyle\frac{am}{\kappa b}}\frac{dp}{\sqrt{(p+p_{12})(p_{12}-p)(p^{2}+p_{12}^{2}-\displaystyle\frac{a\kappa \omega-r}{\kappa b})}}=
\int_{0}^{\xi}d\xi, ~~0<\xi<T_{2},$$
$$-\int_{p}^{0}\sqrt{\displaystyle\frac{am}{\kappa b}}\frac{dp}{\sqrt{(p+p_{12})(p_{12}-p)(p^{2}+p_{12}^{2}-\displaystyle\frac{a\kappa \omega-r}{\kappa b})}}=
\int_{\xi}^{0}d\xi, ~~-T_{2}<\xi<0,$$
 which can be rewritten as
 $$
 \int_{0}^{p}\sqrt{\displaystyle\frac{am}{\kappa b}}\frac{dp}{\sqrt{(p+p_{12})(p_{12}-p)(p^{2}+p_{12}^{2}-\displaystyle\frac{a\kappa \omega-r}{\kappa
 b})}}=\mid\xi\mid,~~-T_{2}<\xi<T_{2}.
 $$
Noting that
$$
  \int_{0}^{p}\frac{dp}{\sqrt{(p+p_{12})(p_{12}-p)(p^{2}+p_{12}^{2}-\displaystyle\frac{a\kappa \omega-r}{\kappa b})}}=g\cdot sn^{-1}\bigg(\sqrt{\frac{p^{2}(2\kappa
  bp^{2}_{12}-a\kappa\omega+r)}{p_{12}^{2}(\kappa bp^{2}+\kappa bp_{12}^{2}-a\kappa\omega+r)}},k\bigg),
$$
where $g=\sqrt{\displaystyle\frac{\kappa b}{2\kappa bp_{12}^{2}-a\kappa\omega+r}}$, $k^{2}=\displaystyle\frac{\kappa bp_{12}^{2}}{2\kappa bp_{12}^{2}-a\kappa\omega+r}$, we
obtain the third type of periodic solution of system (\ref{8})
$$ p_{b_{5}}(\xi)=\sqrt{\displaystyle\frac{(\kappa bp_{12}^{2}-a\kappa\omega+r)p_{12}^{2}sn^{2}(\sqrt{\frac{2\kappa bp_{12}^{2}-a\kappa\omega+r}{am}}\xi)}{2\kappa
bp_{12}^{2}-a\kappa\omega+r-\kappa bp_{12}^{2}sn^{2}(\sqrt{\frac{2\kappa bp_{12}^{2}-a\kappa\omega+r}{am}}\xi)}},~~-T_{2}<\xi<T_{2}.$$

\textbf{3.} When $\kappa b(a\kappa \omega-r)\leq0$ and $\displaystyle\frac{\kappa b}{am}>0$, one can check that the periodic orbits shown in Figs. {\ref{fig2b}} and {\ref{fig3b}} have the same form as solution $p_{b_{5}}(\xi)$. We ignore them here for simplicity.
\subsection{Unbounded solutions of system (\ref{8})}
\par In this subsection, we need to consider two cases to get unbounded solutions of system (\ref{8}).
\par\textbf{1.} When $\kappa b(a\kappa \omega-r)>0$ and $\displaystyle\frac{\kappa b}{am}<0$, we have five subcases to  discuss.
\par\noindent(1) Consider the first type of unbounded orbits, for example $\Gamma^{1}$ and $\Gamma_{1}$, shown in Fig. {\ref{fig1a}}, whose energy $h_{0}$ is higher than energy of saddles $P_{1}$ and $P_{3}$. They can be expressed respectively by
\begin{equation*}
y=\pm\sqrt{-\displaystyle\frac{\kappa b}{am}p^{4}+\frac{a\kappa\omega-r}{am}p^{2}+2h_{0}},
\end{equation*}
where $0<p<+\infty$. For the sake of simplicity, we take $\Gamma_{1}$ for example to calculate its corresponding solution. Choosing initial value $p(0)=+\infty$, we have
$$
-\int_{+\infty}^{p}\sqrt{-\frac{am}{\kappa b}}\frac{dp}{\sqrt{p^{4}+\displaystyle\frac{r-a\kappa \omega}{\kappa b}p^{2}-\frac{2amh_{0}}{\kappa b}}}=\int_{0}^{\xi}d\xi , ~~\xi>0.
$$
Noting that
$$
  \int_{p}^{+\infty}\frac{dp}{\sqrt{p^{4}+\displaystyle\frac{r-a\kappa \omega}{\kappa b}p^{2}-\frac{2amh_{0}}{\kappa b}}}=g\cdot cn^{-1}(\frac{p^{2}-\sqrt{-\frac{2amh_{0}}{\kappa b}}}{p^{2}+\sqrt{-\frac{2amh_{0}}{\kappa b}}},k),
$$
where $g=\displaystyle\frac{1}{2\sqrt[4]{-\frac{2amh_{0}}{\kappa b}}}$, $k^{2}=\displaystyle\frac{2\sqrt{-2am\kappa bh_{0}}+a\kappa \omega-r}{4\sqrt{-2am\kappa bh_{0}}}$, we obtain the first type of unbounded solution of system (\ref{8})
$$
p_{u_{0}}(\xi)=\sqrt[4]{-\frac{2amh_{0}}{\kappa b}}\cdot\sqrt{-1+\frac{2}{1-cn(2\sqrt[4]{-\frac{2\kappa bh_{0}}{am}}\xi)}},~~0<\xi<\xi_{0},
$$
where $\xi_{0}=2\sqrt[4]{-\displaystyle\frac{am}{2\kappa bh_{0}}}\cdot\displaystyle\int_{0}^{\frac{\pi}{2}}{\frac{d\theta}{\sqrt{1-\frac{2\sqrt{-2am\kappa bh_{0}}+a\kappa \omega-r}{4\sqrt{-2am\kappa bh_{0}}}\cdot\sin^{2}\theta}}}.$\\
It is not difficult to check that the corresponding unbounded solution of $\Gamma^{1}$ has the same form as $p_{u_{0}}(\xi)$.

\par\noindent(2) Consider the second type of unbounded orbits $\Gamma^{2}$, $\Gamma_{2}$, $\Gamma^{3}$ and $\Gamma_{3}$ shown in Fig. {\ref{fig1a}}, whose energy is equal to energy of saddles $P_{1}$ and $P_{3}$. The orbits $\Gamma^{2}$ and $\Gamma_{2}$ can be expressed respectively by
\begin{equation*}
  y=\pm\sqrt{-\displaystyle\frac{\kappa b}{am}}\sqrt{(p-\sqrt{\frac{a\kappa \omega-r}{2\kappa b}})^{2}(p+\sqrt{\frac{a\kappa \omega-r}{2\kappa b}})^{2}},
\end{equation*}
where $0<\sqrt{\displaystyle\frac{a\kappa \omega-r}{2\kappa b}}<p<+\infty$. Similar to the discussion above, we only need to discuss the orbit $\Gamma_{2}$. Choosing initial value $p(0)=+\infty$, we have
$$
-\int_{+\infty}^{p}\sqrt{-\frac{am}{\kappa b}}\frac{dp}{{(p-\sqrt{\frac{a\kappa \omega-r}{2\kappa b}})(p+\sqrt{\frac{a\kappa \omega-r}{2\kappa b}})}}=\int_{0}^{\xi}d\xi , ~~\xi>0.
$$
Noting that
$$
  \int_{p}^{+\infty}\frac{dp}{{{(p-\sqrt{\frac{a\kappa \omega-r}{2\kappa b}})(p+\sqrt{\frac{a\kappa \omega-r}{2\kappa b}})}}}=\sqrt{\frac{\kappa b}{2(a\kappa\omega -r)}}\ln\Bigg(\frac{p+\sqrt{\frac{a\kappa \omega-r}{2\kappa b}}}{p-\sqrt{\frac{a\kappa \omega-r}{2\kappa b}}}\Bigg),
$$
we obtain the second type of unbounded solution of system (\ref{8})
$$ p_{u_{1}}(\xi)=\sqrt{\frac{a\kappa \omega-r}{2\kappa b}}\cdot\bigg(1+\frac{2}{\exp(\sqrt{\frac{2(r-a\kappa\omega)}{am}}\xi)-1}\bigg),~~\xi>0.$$
One can check that the corresponding unbounded solution of $\Gamma^{2}$ has the same form as $p_{u_{1}}(\xi)$. If choosing initial value $p(0)=-\infty$ and applying similar calculation to the orbit $\Gamma^{3}$, we can get the corresponding unbounded solution of system (\ref{8})
$$
p_{u_{1^{'}}}(\xi)=-\sqrt{\frac{a\kappa \omega-r}{2\kappa b}}\cdot\bigg(1+\frac{2}{\exp(\sqrt{\frac{2(r-a\kappa\omega)}{am}}\xi)-1}\bigg),~~\xi>0,
$$
where $-\infty<p<-\sqrt{\displaystyle \frac{a\kappa \omega-r}{2\kappa b}}<0$. And it is not difficult to conclude that the corresponding unbounded solution of the orbit $\Gamma_{3}$ has the same form as $p_{u_{1^{'}}}(\xi)$.

\par\noindent(3) Consider the third type of unbounded orbits, for example $\Gamma^{+}_{4}$, shown in Fig. {\ref{fig1a}}, whose energy is lower than energy of the saddle $P_{1}$, but higher than energy of center $P_{2}$. It can be expressed by
\begin{equation*}
  y=\pm\sqrt{-\displaystyle\frac{\kappa b}{am}}\sqrt{(p-p_{1^{'}})(p-p_{2^{'}})(p-p_{3^{'}})(p-p_{4^{'}})},
\end{equation*}
where $p_{1^{'}}$, $p_{2^{'}}$, $p_{3^{'}}$ and $p_{4^{'}}$ satisfy the constraint condition $p_{1^{'}}<-\sqrt{\displaystyle \frac{a\kappa \omega-r}{2\kappa b}}<p_{2^{'}} <0<p_{3^{'}}<\sqrt{\displaystyle \frac{a\kappa \omega-r}{2\kappa b}}<p_{4^{'}}<p<+\infty$. Similarly, we only need to consider the lower branch of $\Gamma^{+}_{4}$. Choosing initial value $p(0)=+\infty$, we have
$$
-\int_{+\infty}^{p}\sqrt{-\frac{am}{\kappa b}}\frac{dp}{\sqrt{(p-p_{1^{'}})(p-p_{2^{'}})(p-p_{3^{'}})(p-p_{4^{'}})}}=\int_{0}^{\xi}d\xi , ~~ \xi>0.
$$
Noting that
$$
  \int_{p}^{+\infty}{\frac{dp}{\sqrt{(p-p_{1^{'}})(p-p_{2^{'}})(p-p_{3^{'}})(p-p_{4^{'}})}}}=g\cdot sn^{-1}(\frac{p_{4^{'}}}{p},k),
$$
where $g=\displaystyle\frac{1}{p_{4^{'}}}$, $k^{2}=\displaystyle\frac{p^{2}_{3{'}}}{p_{4^{'}}^{2}}$, we get the third type of unbounded solution of system (\ref{8})
\begin{equation*}
  p_{u_{2}}(\xi)=\displaystyle\frac{p_{4^{'}}}{sn(p_{4^{'}}\sqrt{-\frac{\kappa b}{am}}\xi)}, ~~0<\xi<\xi_{1},
\end{equation*}
where $\xi_{1}=\displaystyle\frac{4}{p_{4^{'}}}\sqrt{-\frac{am}{\kappa
b}}\cdot\displaystyle\int_{0}^{\frac{\pi}{2}}{\frac{d\theta}{\sqrt{1-\frac{p^{2}_{3^{'}}}{p^{2}_{4^{'}}}\cdot\sin^{2}\theta}}}.$\\
If choosing initial value $p(0)=-\infty$  and adopting the similar calculation to the upper branch of unbounded orbit $\Gamma^{-}_{4}$, we can get the corresponding unbounded solution of system
(\ref{8})
\begin{equation*}
  p_{u_{2^{'}}}(\xi)=\displaystyle\frac{p_{1^{'}}}{sn(p_{1^{'}}\sqrt{-\frac{\kappa b}{am}}\xi)}, ~~0<\xi<\xi_{1^{'}},
\end{equation*}
where $-\infty<p<p_{1^{'}}<-\sqrt{\frac{a\kappa \omega-r}{2\kappa b}}<p_{2^{'}}<0<p_{3^{'}}<\sqrt{\frac{a\kappa \omega-r}{2\kappa b}}<p_{4^{'}}$, $\xi_{1^{'}}=-\displaystyle\frac{4}{p_{1^{'}}}\sqrt{-\frac{am}{\kappa
b}}\cdot\displaystyle\int_{0}^{\frac{\pi}{2}}{\frac{d\theta}{\sqrt{1-\frac{p^{2}_{2^{'}}}{p^{2}_{1^{'}}}\cdot\sin^{2}\theta}}}.$
\par\noindent(4) Consider the fourth type of unbounded orbits, for example $\Gamma^{+}_{5}$ and $\Gamma^{-}_{5}$, shown in Fig. {\ref{fig1a}}, whose energy is equal to energy of the center $P_{2}$. The unbounded orbit $\Gamma^{+}_{5}$ can be expressed by
\begin{equation*}
  y=\pm\sqrt{-\displaystyle\frac{\kappa b}{am}}\sqrt{(p+p_{5^{'}})p^{2}(p-p_{5^{'}})},
\end{equation*}
where $p_{5^{'}}$ satisfies the constraint condition $0<\sqrt{\frac{a\kappa \omega-r}{2\kappa b}}<p_{5^{'}}<p<+\infty$.
Choosing initial value $p(0)=+\infty$, we have
$$
-\int_{+\infty}^{p}\sqrt{-\frac{am}{\kappa b}}\frac{dp}{\sqrt{(p+p_{5^{'}})p^{2}(p-p_{5^{'}})}}=\int_{0}^{\xi}d\xi, ~~\xi>0.
$$
Noting that
$$
  \int_{p}^{+\infty}\frac{dp}{p\sqrt{(p+p_{5^{'}})(p-p_{5^{'}})}}= \frac{1}{p_{5^{'}}}\arcsin\frac{p_{5^{'}}}{p},
$$
we obtain the fourth type of unbounded solution of system (\ref{8})
$$
p_{u_{3}}(\xi)=p_{5^{'}}\csc(p_{5^{'}}\sqrt{-\frac{\kappa b}{am}}\xi),~~0<\xi<\xi_{2},
$$
where $\xi_{2}=\displaystyle\frac{2\pi}{p_{5^{'}}}\sqrt{-\frac{am}{\kappa b}}$.
\par\noindent If choosing initial value $p(0)=-\infty$, we can get another unbounded solution of system (\ref{8}) corresponding to the upper branch of unbounded orbit $\Gamma^{-}_{5}$
$$
p_{u_{3^{'}}}(\xi)=p_{5^{'}}\csc(-p_{5^{'}}\sqrt{-\frac{\kappa b}{am}}\xi),~~0<\xi<\xi_{2^{'}},
$$
where $-\infty<p<p_{5^{'}}<-\sqrt{\frac{a\kappa \omega-r}{2\kappa b}}<0$, $\xi_{2^{'}}=-\displaystyle\frac{2\pi}{p_{5^{'}}}\sqrt{-\frac{am}{\kappa b}}$.
\par\noindent(5) Consider the fifth type of unbounded orbit, for example $\Gamma^{+}_{6}$, shown in Fig. {\ref{fig1a}}, whose energy is lower than energy of the center $P_{2}$. It can be expressed by
\begin{equation*}
  y=\pm\sqrt{-\displaystyle\frac{\kappa b}{am}}\sqrt{(p+p_{6^{'}})(p-p_{6^{'}})(p^{2}+p_{6^{'}}^{2}-\frac{a\kappa\omega-r}{\kappa b})},
\end{equation*}
where $p_{6^{'}}$ satisfies the constraint condition $0<\sqrt{\displaystyle\frac{a\kappa\omega-r}{\kappa b}}<p_{6^{'}}<p<+\infty$.
Choosing initial value $p(0)=+\infty$, we
have
$$
-\int_{+\infty}^{p}\sqrt{-\frac{am}{\kappa b}}\frac{dp}{\sqrt{(p+p_{6^{'}})(p-p_{6^{'}})(p^{2}+p_{6^{'}}^{2}-\displaystyle\frac{a\kappa\omega-r}{\kappa
b})}}=\int_{0}^{\xi}d\xi, ~~\xi>0.
$$
Noting that
 $$
 \int_{p}^{+\infty}\frac{dp}{\sqrt{(p+p_{6^{'}})(p-p_{6^{'}})(p^{2}+p_{6^{'}}^{2}-\displaystyle\frac{a\kappa\omega-r}{\kappa b})}}=g\cdot sn^{-1}(\sqrt{\displaystyle\frac{2\kappa
 bp_{6^{'}}^{2}-a\kappa\omega+r}{\kappa bp^{2}+\kappa bp_{6^{'}}^{2}-a\kappa\omega+r}},k),
 $$
 where $g=\displaystyle\sqrt{\frac{\kappa b}{2\kappa bp_{6^{'}}^{2}-a\kappa\omega+r}}$, $k^{2}=\displaystyle\frac{\kappa bp_{6^{'}}^{2}-a\kappa\omega+r}{2\kappa
 bp_{6^{'}}^{2}-a\kappa\omega+r}$,
we obtain the fifth type of unbounded solution of system (\ref{8})
$$
p_{u_{4}}(\xi)=\sqrt{-\frac{\kappa bp_{6^{'}}^{2}-a\kappa\omega+r}{\kappa b}+\frac{2\kappa bp_{6^{'}}^{2}-a\kappa\omega+r}{\kappa bsn^{2}(\sqrt{-\frac{2\kappa
bp_{6^{'}}^{2}-a\kappa\omega+r}{am}}\xi)}},~~0<\xi<\xi_{3},
$$
where $\xi_{3}=\sqrt{-\displaystyle\frac{4am}{2\kappa
bp_{6^{'}}^{2}-a\kappa\omega+r}}\cdot\displaystyle\int_{0}^{\frac{\pi}{2}}{\displaystyle\frac{d\theta}{\sqrt{1-\frac{\kappa bp_{6^{'}}^{2}-a\kappa\omega+r}{2\kappa
bp_{6^{'}}^{2}-a\kappa\omega+r}\cdot\sin^{2}\theta}}}$.
\par\noindent If choosing initial value $p(0)=-\infty$ and applying similar calculation to the upper branch of unbounded orbit $\Gamma^{-}_{6}$, it is obvious to conclude that the corresponding unbounded solution has the same form as unbounded solution $p_{u_{4}}(\xi)$.
\par\textbf{2.} When $\kappa b(a\kappa \omega-r)\leq0$ and $\displaystyle\frac{\kappa b}{am}<0$, we need to consider three subcases in this case.
\par\noindent(1) Consider the first type of unbounded orbits, for example $\Upsilon^{3}$, $\Upsilon_{3}$, $Y^{1}$ and $Y_{1}$, shown in Figs. {\ref{fig2a}} and {\ref{fig3a}}, whose energy $h_{0^{'}}$ is higher than energy of the saddle $P_{2}$. One can check that the corresponding unbounded solution of $\Upsilon_{3}$ has the same form as solution $p_{u_{0}}(\xi)$. We ignore it here for simplicity.

In particular, when $\kappa b(a\kappa \omega-r)=0$, the unbounded orbits $Y^{1}$ and $Y_{1}$ shown in Fig. {\ref{fig3a}} can be expressed respectively by
\begin{equation*}
  y=\pm\sqrt{-\displaystyle\frac{\kappa b}{am}p^{4}+2h_{0^{'}}},
\end{equation*}
where $0<p<+\infty$. Choosing initial value $p(0)=+\infty$, we have
$$
-\int_{+\infty}^{p}\sqrt{-\displaystyle\frac{am}{\kappa b}}\frac{dp}{\sqrt{p^{4}-\displaystyle\frac{2amh_{0^{'}}}{\kappa b}}}=\int_{0}^{\xi}d\xi, ~~\xi>0.
$$
Noting that
 $$
 \int_{p}^{+\infty}\frac{dp}{\sqrt{p^{4}-\displaystyle\frac{2amh_{0^{'}}}{\kappa b}}}=\displaystyle\frac{1}{2}\sqrt[4]{-\displaystyle\frac{\kappa b}{2amh_{0^{'}}}}\cdot cn^{-1}(\frac{p^{2}-1}{p^{2}+1},k),
 $$
where $k^{2}=\displaystyle\frac{1}{2}$, we obtain the sixth type of unbounded solution of system (\ref{8})
$$
p_{u_{5^{'}}}(\xi)=\sqrt{\displaystyle\frac{2}{1-cn(8h_{0^{'}}\xi)}-1},~~0<\xi<\xi_{4^{'}},\\
$$
where $\xi_{4^{'}}=\displaystyle\frac{1}{2h_{0^{'}}}\cdot\displaystyle\int_{0}^{\frac{\pi}{2}}{\frac{d\theta}{\sqrt{1-\frac{1}{2}\cdot\sin^{2}\theta}}}$.\\

\par\noindent(2) Consider the second type of unbounded orbits shown in Figs. {\ref{fig2a}} and {\ref{fig3a}}, whose energy is equal to energy of the saddle $P_{2}$. The unbounded orbits $\Upsilon^{4}$ and $\Upsilon_{4}$ can be expressed by
\begin{equation*}
  y=\pm\sqrt{-\displaystyle\frac{\kappa b}{am}}\sqrt{p^{4}+\frac{r-a\kappa\omega}{\kappa b}p^{2}},
\end{equation*}
where $0<p<+\infty$.
Choosing initial value $p(0)=+\infty$, we have
$$
-\int_{+\infty}^{p}\sqrt{-\frac{am}{\kappa b}}\frac{dp}{p\sqrt{p^{2}+\displaystyle\frac{r-a\kappa\omega}{\kappa b}}}=\int_{0}^{\xi}d\xi, ~~\xi>0.
$$
Noting that
 $$
 \int_{p}^{+\infty}\frac{dp}{p\sqrt{p^{2}+\displaystyle\frac{r-a\kappa\omega}{\kappa b}}}=\sqrt{\frac{\kappa b}{r-a\kappa\omega}}\cdot\ln\bigg(\frac{\sqrt{p^{2}+\frac{r-a\kappa\omega}{\kappa b}}+\sqrt{\frac{r-a\kappa\omega}{\kappa b}}}{p}\bigg),
 $$
we obtain the seventh type of unbounded solution of system (\ref{8})
$$
p_{u_{6}}(\xi)=2\sqrt{\frac{r-a\kappa\omega}{\kappa b}}\cdot\displaystyle\frac{\exp(\sqrt{\frac{a\kappa\omega-r}{am}}\xi)}{\exp(2\sqrt{\frac{a\kappa\omega-r}{am}}\xi)-1},~~\xi>0.\\
$$
Similar calculation can be applied to the unbounded orbit $\Upsilon^{5}$ shown in Fig. {\ref{fig2a}}. If choosing $p(0)=-\infty$ and $-\infty<p<0$, we obtain another unbounded solution of system (\ref{8}) as follows
$$
p_{u_{6^{'}}}(\xi)=2\sqrt{\frac{r-a\kappa\omega}{\kappa b}}\cdot\displaystyle\frac{\exp(\sqrt{\frac{a\kappa\omega-r}{am}}\xi)}{1-\exp(2\sqrt{\frac{a\kappa\omega-r}{am}}\xi)},~~\xi>0.\\
$$
In particular, when $\kappa b(a\kappa \omega-r)=0$, the unbounded orbits $Y^{2}$ and $Y_{2}$ shown in Fig. {\ref{fig3a}} can be expressed respectively by
\begin{equation*}
  y=\pm\sqrt{-\displaystyle\frac{\kappa b}{am}}p^{2},
\end{equation*}
where $0<p<+\infty$. By a direct calculation, we obtain the corresponding unbounded solution of system (\ref{8})
$$
p_{u_{7}}(\xi)=\sqrt{-\frac{am}{\kappa b}}\cdot\frac{1}{\xi},~~\xi>0.\\
$$
If choosing $p(0)=-\infty$ and adopting the similar calculation to the unbounded orbit $Y^{3}$ shown in Fig. {\ref{fig3a}}, we have
$$
p_{u_{7^{'}}}(\xi)=-\sqrt{-\frac{am}{\kappa b}}\cdot\frac{1}{\xi},~~\xi>0,\\
$$
where $-\infty<p<0$.
\par\noindent(3) Consider the third type of unbounded orbits, for example $\Upsilon^{+}_{6}$, $\Upsilon^{-}_{6}$, $Y^{+}_{4}$ and $Y^{-}_{4}$, shown in Figs. {\ref{fig2a}} and {\ref{fig3a}}, whose energy is lower than energy of the saddle $P_{2}$. It is not difficult to conclude that the corresponding unbounded solutions have the same form as solution $p_{u_{4}}(\xi)$. Due to the the tedious expressions of the solutions, we ignore them here for simplicity.
\subsection{Bounded solutions of system (\ref{13})}
\par From the energy function
$H(\phi,y)=\displaystyle\frac{1}{2}y^{2}-\displaystyle\frac{\alpha_{1}}{4}\phi^{4}-\displaystyle\frac{\alpha_{2}}{2}\phi^{2}+\displaystyle\frac{\alpha_{3}}{2}\frac{1}{\phi^{2}}=h$
and the first equation of system (\ref{13}), we can obtain the following expression
\begin{equation}\label{18}
\begin{aligned}
\xi&=\int_{\phi_{0}}^{\phi}\frac{\phi d\phi}{\sqrt{\frac{\alpha_{1}}{2}\phi^{6}+\alpha_{2}\phi^{4}+2h\phi^{2}-\alpha_{3}}}\\
&=\int_{\psi_{0}}^{\psi}\frac{d\psi}{2\sqrt{\frac{\alpha_{1}}{2}\psi^{3}+\alpha_{2}\psi^{2}+2h\psi-\alpha_{3}}},
\end{aligned}
\end{equation}
where $\phi>0$ and $\psi=\phi^{2}$. Due to the symmetry of system (\ref{13}), it is easy to check that the bounded solutions of system (\ref{13}) where $\phi<0$ have the same expressions of bounded solutions of system (\ref{13}) where $\phi>0$ except for the difference of signs.
\par\textbf{I.} When $\alpha_{1}>0$, $\alpha_{2}<0$ and $0<\alpha_{3}<\displaystyle-\frac{4\alpha_{2}^{3}}{27\alpha_{1}^{2}}$, we have two subcases in this case.
\par\noindent(i) Consider the periodic orbits shown in Fig. {\ref{fig4a}}, whose energy is lower than energy of the saddle $\widetilde{P}_{3}$, but higher than energy of center $\widetilde{P}_{1}$. Assuming that $r_{1}$, $r_{2}$ and $r_{3}$ satisfy the constraint condition $0<r_{1}<\psi<r_{2}<r_{3}$, the period is $2T_{0^{'}}$ and choosing $\psi_{0}=\psi(0)=r_{1}$, we have
$$
\int_{r_{1}}^{\psi}{\frac{d\psi}{\sqrt{2\alpha_{1}}\sqrt{(\psi-r_{1})(r_{2}-\psi)(r_{3}-\psi)}}}=\int_{0}^{\xi}{d\xi},~~0<\xi<T_{0^{'}},
$$
$$
-\int_{r_{1}}^{\psi}{\frac{d\psi}{\sqrt{2\alpha_{1}}\sqrt{(\psi-r_{1})(r_{2}-\psi)(r_{3}-\psi)}}}=\int_{\xi}^{0}{d\xi},~~-T_{0^{'}}<\xi<0,
$$
which can be rewritten as

$$
\int_{r_{1}}^{\psi}{\frac{d\psi}{\sqrt{2\alpha_{1}}\sqrt{(\psi-r_{1})(r_{2}-\psi)(r_{3}-\psi)}}}=\mid\xi\mid, ~~-T_{0^{'}}<\xi<T_{0^{'}}.
$$
Noting that
$$
  \int_{r_{1}}^{\psi}{\frac{d\psi}{\sqrt{(\psi-r_{1})(r_{2}-\psi)(r_{3}-\psi)}}}=g\cdot sn^{-1}(\sqrt{\frac{\psi-r_{1}}{r_{2}-r_{1}}},k),
$$
where $g=\displaystyle\frac{2}{\sqrt{r_{3}-r_{1}}}$, $k^{2}=\displaystyle\frac{r_{2}-r_{1}}{r_{3}-r_{1}}$, we get the first type of periodic solution of system (\ref{13})
\begin{equation*}
  \phi_{b_{1}}(\xi)=\sqrt{r_{1}+(r_{2}-r_{1})sn^{2}(\sqrt{\frac{\alpha_{1}(r_{3}-r_{1})}{2}}\xi)},~~-T_{0^{'}}<\xi<T_{0^{'}}.\\
\end{equation*}
\par\noindent(ii) Consider the homoclinic orbit $\Pi^{+}_{0}$ shown in Fig. {\ref{fig4a}}, whose energy is equal to energy of the saddle $\widetilde{P}_{3}$. Assuming that $r_{4}$ and $r_{5}$ satisfy the constraint condition $0<r_{4}<\psi<r_{5}$ and choosing $\psi_{0}=\psi(0)=r_{4}$, we have
\begin{equation*}
\begin{aligned}
\int_{r_{4}}^{\psi}{\frac{d\psi}{\sqrt{2\alpha_{1}}\sqrt{(\psi-r_{4})(r_{5}-\psi)^{2}}}}
&=\int_{0}^{\xi}{d\xi},~~\xi>0,\\
-\int_{\psi}^{r_{4}}{\frac{d\psi}{\sqrt{2\alpha_{1}}\sqrt{(\psi-r_{4})(r_{5}-\psi)^{2}}}}
&=\int_{\xi}^{0}{d\xi},~~\xi<0,
\end{aligned}
\end{equation*}
which can be rewritten as
$$
  \int_{r_{4}}^{\psi}{\frac{d\psi}{\sqrt{2\alpha_{1}}(r_{5}-\psi)\sqrt{\psi-r_{4}}}}=\mid\xi\mid,~~-\infty<\xi<+\infty.
$$
Noting that
$$
  \int_{r_{4}}^{\psi}{\frac{d\psi}{(r_{5}-\psi)\sqrt{\psi-r_{4}}}}=-\frac{1}{\sqrt{r_{5}-r_{4}}}\ln{\frac{\sqrt{r_{5}-r_{4}}-\sqrt{\psi-r_{4}}}{\sqrt{r_{5}-r_{4}}+\sqrt{\psi-r_{4}}}},
$$
we obtain the expression of solitary wave solution of system (\ref{13})
$$
\phi_{b_{2}}(\xi)=\sqrt{r_{4}+ \frac{(r_{5}-r_{4})(1-{\rm exp}(\sqrt{2\alpha_{1}(r_{5}-r_{4})}\xi))^{2}}{(1+{\rm exp}(\sqrt{2\alpha_{1}(r_{5}-r_{4})}\xi))^{2}}},~~-\infty<\xi<+\infty.\\
$$

\textbf{II.} When $\alpha_{1}<0$, there only exist two families of periodic orbits shown in Fig. {\ref{fig4d}}. Assuming that $r_{6}$ and $r_{7}$  satisfy the constraint conditions $0<r_{6}<\psi<r_{7}$, $r_{6}+r_{7}+\displaystyle\frac{2\alpha_{2}}{\alpha_{1}}>0$, the period is $2T_{1^{'}}$ and choosing $\psi_{0}=\psi(0)=r_{6}$, we have
\begin{equation*}
\begin{aligned}
\int_{r_{6}}^{\psi}{\frac{d\psi}{\sqrt{-2\alpha_{1}}\sqrt{(\psi-r_{6})(r_{7}-\psi)(\psi+r_{6}+r_{7}+\frac{2\alpha_{2}}{\alpha_{1}})}}}
&=\int_{0}^{\xi}{d\xi},~~0<\xi<T_{1^{'}},\\
-\int_{\psi}^{r_{6}}{\frac{d\psi}{\sqrt{-2\alpha_{1}}\sqrt{(\psi-r_{6})(r_{7}-\psi)(\psi+r_{6}+r_{7}+\frac{2\alpha_{2}}{\alpha_{1}})}}}
&=\int_{\xi}^{0}{d\xi},~~-T_{1^{'}}<\xi<0,
\end{aligned}
\end{equation*}
which can be rewritten as
$$
  \int_{r_{6}}^{\psi}{\frac{d\psi}{\sqrt{-2\alpha_{1}}\sqrt{(\psi-r_{6})(r_{7}-\psi)(\psi+r_{6}+r_{7}+\frac{2\alpha_{2}}{\alpha_{1}})}}}
  =\mid\xi\mid,~~-T_{1^{'}}<\xi<T_{1^{'}}.
$$
Noting that
$$
  \int_{r_{6}}^{\psi}{\frac{d\psi}{\sqrt{(\psi-r_{6})(r_{7}-\psi)(\psi+r_{6}+r_{7}+\frac{2\alpha_{2}}{\alpha_{1}})}}}=g\cdot sn^{-1}(\sqrt{\frac{(2r_{7}+r_{6}+\frac{2\alpha_{2}}{\alpha_{1}})(\psi-r_{6})}{(r_{7}-r_{6})(\psi+r_{6}+r_{7}+\frac{2\alpha_{2}}{\alpha_{1}})}},k),
$$
where $g=\displaystyle\frac{2}{\sqrt{2r_{7}+r_{6}+\frac{2\alpha_{2}}{\alpha_{1}}}}$, $k^{2}=\displaystyle\frac{r_{7}-r_{6}}{2r_{7}+r_{6}+\frac{2\alpha_{2}}{\alpha_{1}}}$,
we obtain the second type of periodic solution of system (\ref{13})
$$
\phi_{b_{3}}(\xi)=\sqrt{r_{6}+\displaystyle\frac{(2\alpha_{1}r_{6}+\alpha_{1}r_{7}+2\alpha_{2})(r_{7}-r_{6})sn^{2}(\sqrt{-\frac{
2\alpha_{1}r_{7}+\alpha_{1}r_{6}+2\alpha_{2}}{2}}\xi)}{\alpha_{1}(r_{7}-r_{6})
sn^{2}(\sqrt{-\frac{2\alpha_{1}r_{7}+\alpha_{1}r_{6}+2\alpha_{2}}{2}}\xi)-(2\alpha_{1}r_{7}+\alpha_{1}r_{6}+2\alpha_{2})}},~~-T_{1^{'}}<\xi<T_{1^{'}}.
$$
\subsection{Unbounded solutions of system (\ref{13})}
\par In this subsection, to seek the unbounded solutions of system (\ref{13}) where $\psi>0$ and $\phi=\sqrt{\psi}>0$,  we need to discuss three cases. If applying similar calculation to the unbounded orbits of system (\ref{13}) where $\phi<0$, one can check the corresponding unbounded solutions have the same expressions as unbounded solutions where $\phi>0$ except for the difference of signs.
\par\textbf{I.} When $\alpha_{1}>0$, $\alpha_{2}<0$ and $0<\alpha_{3}<\displaystyle-\frac{4\alpha_{2}^{3}}{27\alpha_{1}^{2}}$, we have five subcases in this case.
\par\noindent(i) Consider the first type of unbounded orbits, for example $\Pi_{1}$, shown in Fig. {\ref{fig4a}}, whose energy is higher than energy of the saddle $\widetilde{P}_{3}$. We consider to discuss the lower branch of orbit $\Pi_{1}$ for simplicity. Assuming that $r_{1^{'}}$ satisfies the constraint condition $0<r_{1^{'}}<\psi<+\infty$, $h_{1}\in(h(\widetilde{P}_{3}),+\infty)$ and choosing
$\psi_{0}=\psi(0)=+\infty$, we have
$$
-\int_{+\infty}^{\psi}{\frac{d\psi}{\sqrt{2\alpha_{1}}\sqrt{(\psi-r_{1^{'}})[\psi^{2}+(\frac{2\alpha_{2}}{\alpha_{1}}+r_{1^{'}})\psi+r_{1^{'}}+\frac{4h_{1}}{\alpha_{1}}]}}}=\int_{0}^{\xi}{d\xi},~~\xi>0.
$$
Noting that
$$
\int_{\psi}^{+\infty}{\frac{d\psi}{\sqrt{(\psi-r_{1^{'}})[\psi^{2}+(\frac{2\alpha_{2}}{\alpha_{1}}+r_{1^{'}})\psi+r_{1^{'}}+\frac{4h_{1}}{\alpha_{1}}]}}}=g\cdot
cn^{-1}(\frac{\psi-r_{1^{'}}-A}{\psi-r_{1^{'}}+A},k),
$$
where $A^{2}=\displaystyle\frac{2\alpha_{1}r^{2}_{1^{'}}+(\alpha_{1}+2\alpha_{2})r_{1^{'}}+4h_{1}}{\alpha_{1}}$, $g=\displaystyle\frac{1}{\displaystyle\sqrt{A}}=\sqrt{\displaystyle\frac{\alpha_{1}}{2\alpha_{1}r_{1^{'}}^{2}+(\alpha_{1}+2\alpha_{2})r_{1^{'}}+4h_{1}}}$ and
$k^{2}=\displaystyle\frac{\sqrt{8\alpha_{1}^{2}r_{1^{'}}^{2}+4\alpha_{1}(\alpha_{1}+2\alpha_{2})r_{1^{'}}+16\alpha_{1}h_{1}}-(2\alpha_{2}+3\alpha_{1}r_{1^{'}})}{\sqrt{32\alpha^{2}_{1}r_{1^{'}}^{2}+16\alpha_{1}(\alpha_{1}+\alpha_{2})r_{1^{'}}+64\alpha_{1}h_{1}}}$,
we get the first type of unbounded solution of system (\ref{13})
\begin{equation*}
  \phi_{u_{1}}(\xi)=\sqrt{r_{1^{'}}-\sqrt[4]{\frac{2\alpha_{1}r_{1^{'}}^{2}+(\alpha_{1}+2\alpha_{2})r_{1^{'}}+4h_{1}}{\alpha_{1}}}
  +\frac{2\sqrt[4]{\frac{2\alpha_{1}r_{1^{'}}^{2}+(\alpha_{1}+2\alpha_{2})r_{1^{'}}+4h_{1}}{\alpha_{1}}}}
  {1-cn(\sqrt{2\alpha_{1}}\sqrt[4]{\frac{2\alpha_{1}r_{1^{'}}^{2}+(\alpha_{1}+2\alpha_{2})r_{1^{'}}+4h_{1}}{\alpha_{1}}}\xi)}},\\
\end{equation*}
where $0<\xi<\xi_{0^{''}}$ and\\
$\xi_{0^{''}}=\displaystyle\frac{4}{\sqrt{2\alpha_{1}}\sqrt[4]{\frac{2\alpha_{1}r_{1^{'}}^{2}+(\alpha_{1}+2\alpha_{2})r_{1^{'}}+4h_{1}}
{\alpha_{1}}}}\cdot\int_{0}^{\frac{\pi}{2}}{\frac{d\theta}{\sqrt{1-\frac{\sqrt{8\alpha_{1}^{2}r_{1^{'}}^{2}+4\alpha_{1}(\alpha_{1}+2\alpha_{2})r_{1^{'}}+16\alpha_{1}h_{1}}-(2\alpha_{2}+3\alpha_{1}r_{1^{'}})}{\sqrt{32\alpha^{2}_{1}r_{1^{'}}^{2}+16\alpha_{1}(\alpha_{1}+\alpha_{2})r_{1^{'}}+64\alpha_{1}h_{1}}}\cdot
\sin^{2}\theta}}}.$\\
Similar calculation can be applied to the upper branch of the orbit $\Pi_{1}$. And one can check the corresponding solution has the same form as solution $\phi_{u_{1}}(\xi)$.

\par\noindent(ii) Consider the second type of unbounded orbit $\Pi_{2}$ shown in Fig. {\ref{fig4a}}, whose energy is equal to energy of the saddle $\widetilde{P}_{3}$. Similar to the discussion above, we only need to discuss the lower branch of orbit $\Pi_{2}$. Assuming that $r_{2^{'}}$ and $r_{3^{'}}$  satisfy the constraint condition $0<r_{2^{'}}<r_{3^{'}}<\psi<+\infty$ and choosing $\psi_{0}=\psi(0)=+\infty$, we have
\begin{equation*}
\begin{aligned}
-\int_{+\infty}^{\psi}{\frac{d\psi}{\sqrt{2\alpha_{1}}\sqrt{(\psi-r_{2^{'}})(\psi-r_{3^{'}})^{2}}}}
&=\int_{0}^{\xi}{d\xi},~~\xi>0.
\end{aligned}
\end{equation*}
Noting that
$$
  \int_{\psi}^{+\infty}{\frac{d\psi}{(\psi-r_{3^{'}})\sqrt{\psi-r_{2^{'}}}}}=
  -\frac{1}{\sqrt{r_{3^{'}}-r_{2^{'}}}}\ln{\frac{\sqrt{\psi-r_{2^{'}}}-\sqrt{r_{3^{'}}-r_{2^{'}}}}{\sqrt{\psi-r_{2^{'}}}+\sqrt{r_{3^{'}}-r_{2^{'}}}}},
$$
we get the second type of unbounded solution of system (\ref{13})
$$
\phi_{u_{2}}(\xi)=\sqrt{r_{2^{'}}+ \frac{(r_{3^{'}}-r_{2^{'}})(1+{\rm exp}(\sqrt{2\alpha_{1}(r_{3^{'}}-r_{2^{'}})}\xi))^{2}}{(1-{\rm
exp}(\sqrt{2\alpha_{1}(r_{3^{'}}-r_{2^{'}})}\xi))^{2}}},~~\xi>0.
$$
\par\noindent(iii) Consider the third type of unbounded orbits, for example unbounded orbit $\Pi_{3}$, shown in Fig. {\ref{fig4a}}, whose energy is higher than energy of the center $\widetilde{P}_{1}$, but lower than energy of saddle $\widetilde{P}_{3}$.
Assuming that $r_{4^{'}}$, $r_{5^{'}}$ and $r_{6^{'}}$ satisfy the
constraint condition $0<r_{4^{'}}<r_{5^{'}}<r_{6^{'}}<\psi<+\infty$ and choosing $\psi_{0}=\psi(0)=+\infty$, it can be expressed by
$$
-\int_{+\infty}^{\psi}{\frac{d\psi}{\sqrt{2\alpha_{1}}\sqrt{(\psi-r_{4^{'}})(\psi-r_{5^{'}})(\psi-r_{6^{'}})}}}=\int_{0}^{\xi}{d\xi},~~\xi>0.
$$
Noting that
$$
  \int_{\psi}^{+\infty}{\frac{d\psi}{\sqrt{(\psi-r_{4^{'}})(\psi-r_{5^{'}})(\psi-r_{6^{'}})}}}=g\cdot sn^{-1}(\sqrt{\frac{r_{6^{'}}-r_{4^{'}}}{\psi-r_{4^{'}}}},k),
$$
where $g=\displaystyle\frac{2}{\sqrt{r_{6^{'}}-r_{4^{'}}}}$, $k^{2}=\displaystyle\frac{r_{5^{'}}-r_{4^{'}}}{r_{6^{'}}-r_{4^{'}}}$, we get the third type of unbounded solution of
system (\ref{13})
\begin{equation*}
  \phi_{u_{3}}(\xi)=\sqrt{r_{4^{'}}+\frac{r_{6^{'}}-r_{4^{'}}}{sn^{2}(\sqrt{\displaystyle\frac{\alpha_{1}(r_{6^{'}}-r_{4^{'}})}{2}}\xi)}},~~0<\xi<\xi_{1^{''}},\\
\end{equation*}
where $\xi_{1^{''}}=\sqrt{\displaystyle\frac{8}{\alpha_{1}(r_{6^{'}}-r_{4^{'}})}}\cdot\displaystyle\int_{0}^{\frac{\pi}{2}}{\displaystyle\frac{d\theta}
{\sqrt{1-\frac{r_{5^{'}}-r_{4^{'}}}{r_{6^{'}}-r_{4^{'}}}\cdot\sin^{2}\theta}}}.$
\par\noindent(iv) Consider the fourth type of unbounded orbits, for example $\Pi_{4}$, shown in Fig. {\ref{fig4a}}, whose energy is equal to energy of the center $\widetilde{P}_{2}$. 
 Assuming that $r_{7^{'}}$ and $r_{8^{'}}$ satisfy the constraint condition $0<r_{7^{'}}<r_{8^{'}}<\psi<+\infty$ and choosing $\psi_{0}=\psi(0)=+\infty$, we have
\begin{equation*}
\begin{aligned}
-\int_{+\infty}^{\psi}{\frac{d\psi}{\sqrt{2\alpha_{1}}\sqrt{(\psi-r_{7^{'}})^{2}(\psi-r_{8^{'}})}}}
&=\int_{0}^{\xi}{d\xi},~~\xi>0.
\end{aligned}
\end{equation*}
Noting that
$$
  \int_{\psi}^{+\infty}{\frac{d\psi}{(\psi-r_{7^{'}})\sqrt{\psi-r_{8^{'}}}}}=
  \frac{1}{\sqrt{r_{8^{'}}-r_{7^{'}}}}(\pi-2\arctan\sqrt{\frac{\psi-r_{8^{'}}}{r_{8^{'}}-r_{7^{'}}}}),
$$
we obtain the fourth type of unbounded solution of system (\ref{13})
$$
\phi_{u_{4}}(\xi)=\sqrt{r_{8^{'}}+(r_{8^{'}}-r_{7^{'}})\cdot \cot^{2}(\sqrt{\frac{\alpha_{1}(r_{8^{'}}-r_{7^{'}})}{2}}\xi)} ,~~0<\xi<\xi_{2^{''}},
$$
where $\xi_{2^{''}}=\sqrt{\displaystyle\frac{2}{\alpha_{1}(r_{8^{'}}-r_{7^{'}})}}\cdot\pi$.
\par\noindent(v) Consider the fifth type of unbounded orbits, for example $\Pi_{5}$, shown in Fig. {\ref{fig4a}}, whose energy is lower than energy of the center $\widetilde{P}_{2}$. It is not difficult to obtain the fifth type of unbounded solution as follows
 \begin{equation*}
  \phi_{u_{5}}(\xi)=\sqrt{r_{9^{'}}-\sqrt[4]{\frac{2\alpha_{1}r_{9^{'}}^{2}+(\alpha_{1}+2\alpha_{2})r_{9^{'}}+4h_{2}}{\alpha_{1}}}
  +\frac{2\sqrt[4]{\frac{2\alpha_{1}r_{9^{'}}^{2}+(\alpha_{1}+2\alpha_{2})r_{9^{'}}+4h_{2}}{\alpha_{1}}}}
  {1-cn(\sqrt{2\alpha_{1}}\sqrt[4]{\frac{2\alpha_{1}r_{9^{'}}^{2}+(\alpha_{1}+2\alpha_{2})r_{9^{'}}+4h_{2}}{\alpha_{1}}}\xi)}},
\end{equation*}
where $0<r_{9^{'}}<\psi<+\infty$, $h_{2}\in(0,h(\widetilde{P}_{2}))$, $0<\xi<\xi_{3^{''}}$ and \\
$\xi_{3^{''}}=\displaystyle\frac{4}{\sqrt{2\alpha_{1}}\sqrt[4]{\frac{2\alpha_{1}r_{9^{'}}^{2}+(\alpha_{1}+2\alpha_{2})r_{9^{'}}+4h_{2}}
{\alpha_{1}}}}\cdot\int_{0}^{\frac{\pi}{2}}{\frac{d\theta}{\sqrt{1-\frac{\sqrt{8\alpha_{1}^{2}r_{9^{'}}^{2}+4\alpha_{1}(\alpha_{1}+2\alpha_{2})r_{9^{'}}+16\alpha_{1}h_{2}}-(2\alpha_{2}+3\alpha_{1}r_{9^{'}})}{\sqrt{32\alpha^{2}_{1}r_{9^{'}}^{2}+16\alpha_{1}(\alpha_{1}+\alpha_{2})r_{9^{'}}+64\alpha_{1}h_{2}}}\cdot
\sin^{2}\theta}}}.$
\par\textbf{II.} When $\alpha_{1}>0$, $\alpha_{2}<0$ and $\alpha_{3}=\displaystyle-\frac{4\alpha_{2}^{3}}{27\alpha_{1}^{2}}$, we consider two subcases in this case.
\par\noindent(i) Consider the first type of unbounded orbits $\Omega^{1}$ and $\Omega_{1}$ shown in Fig. {\ref{fig4b}}, whose energy is equal to energy of the cusp $\widetilde{P}_{5}$.
Choosing $\psi_{0}=\psi(0)=+\infty$, we have
$$
-\int_{+\infty}^{\psi}{\frac{d\psi}{\sqrt{2\alpha_{1}}\sqrt{(\psi+\frac{2\alpha_{2}}{3\alpha_{1}})^{2}(\psi+\frac{2\alpha_{2}}{3\alpha_{1}})}}}=\int_{0}^{\xi}{d\xi},~~\xi>0.
$$
Noting that
$$
\int_{\psi}^{+\infty}{\frac{d\psi}{(\psi+\frac{2\alpha_{2}}{3\alpha_{1}})\sqrt{\psi+\frac{2\alpha_{2}}{3\alpha_{1}}}}}=\frac{2}{\sqrt{\psi+\frac{2\alpha_{2}}{3\alpha_{1}}}},
$$
we get the sixth type of unbounded solution of system (\ref{13})
\begin{equation*}
  \phi_{u_{6}}(\xi)=\sqrt{-\frac{2\alpha_{2}}{3\alpha_{1}}+\frac{2}{\alpha_{1}\xi^{2}}},~~\xi>0.\\
\end{equation*}

\par\noindent(ii) Consider other unbounded orbits, for example $\Omega_{2}$ and $\Omega_{3}$, shown in Fig. {\ref{fig4b}}, they can be expressed uniformly. Assuming that $r_{10^{'}}$ satisfies the constraint condition $0<r_{10^{'}}<\psi<+\infty$, $h_{3}\neq h(\widetilde{P}_{5})$ and choosing $\psi_{0}=\psi(0)=+\infty$, we obtain
\begin{equation*}
  \phi_{u_{7}}(\xi)=\sqrt{r_{10^{'}}-\sqrt[4]{\frac{2\alpha_{1}r_{10^{'}}^{2}+(\alpha_{1}+2\alpha_{2})r_{10^{'}}+4h_{3}}{\alpha_{1}}}
  +\frac{2\sqrt[4]{\frac{2\alpha_{1}r_{10^{'}}^{2}+(\alpha_{1}+2\alpha_{2})r_{10^{'}}+4h_{3}}{\alpha_{1}}}}
  {1-cn(\sqrt{2\alpha_{1}}\sqrt[4]{\frac{2\alpha_{1}r_{10^{'}}^{2}+(\alpha_{1}+2\alpha_{2})r_{10^{'}}+4h_{3}}{\alpha_{1}}}\xi)}},\\
\end{equation*}
where $0<\xi<\xi_{4^{''}}$ and\\
$\xi_{4^{''}}=\displaystyle\frac{4}{\sqrt{2\alpha_{1}}\sqrt[4]{\frac{2\alpha_{1}r_{10^{'}}^{2}+(\alpha_{1}+2\alpha_{2})r_{10^{'}}+4h_{3}}
{\alpha_{1}}}}\cdot\int_{0}^{\frac{\pi}{2}}{\frac{d\theta}{\sqrt{1-\frac{\sqrt{8\alpha_{1}^{2}r_{10^{'}}^{2}+4\alpha_{1}(\alpha_{1}+2\alpha_{2})r_{10^{'}}+16\alpha_{1}h_{3}}-(2\alpha_{2}+3\alpha_{1}r_{10^{'}})}{\sqrt{32\alpha^{2}_{1}r_{10^{'}}^{2}+16\alpha_{1}(\alpha_{1}+\alpha_{2})r_{10^{'}}+64\alpha_{1}h_{3}}}\cdot
\sin^{2}\theta}}}.$\\

\par\textbf{III.} Simlarly, when $\alpha_{1}>0$, $\alpha_{2} \geq 0$, one can check that the unbounded solution $\phi_{u_{8}}(\xi)$ has the same form as solution $\phi_{u_{1}}(\xi)$. We get the corresponding unbounded solution of system (\ref{13})
\begin{equation*}
  \phi_{u_{8}}(\xi)=\sqrt{r_{11^{'}}-\sqrt[4]{\frac{2\alpha_{1}r_{11^{'}}^{2}+(\alpha_{1}+2\alpha_{2})r_{11^{'}}+4h_{4}}{\alpha_{1}}}
  +\frac{2\sqrt[4]{\frac{2\alpha_{1}r_{11^{'}}^{2}+(\alpha_{1}+2\alpha_{2})r_{11^{'}}+4h_{4}}{\alpha_{1}}}}
  {1-cn(\sqrt{2\alpha_{1}}\sqrt[4]{\frac{2\alpha_{1}r_{11^{'}}^{2}+(\alpha_{1}+2\alpha_{2})r_{11^{'}}+4h_{4}}{\alpha_{1}}}\xi)}},\\
\end{equation*}
where $0<r_{11^{'}}<\psi<+\infty$, $h_{4}$ is a real number, $0<\xi<\xi_{5^{''}}$ and\\
$\xi_{5^{''}}=\displaystyle\frac{4}{\sqrt{2\alpha_{1}}\sqrt[4]{\frac{2\alpha_{1}r_{11^{'}}^{2}+(\alpha_{1}+2\alpha_{2})r_{11^{'}}+4h_{4}}
{\alpha_{1}}}}\cdot\int_{0}^{\frac{\pi}{2}}{\frac{d\theta}{\sqrt{1-\frac{\sqrt{8\alpha_{1}^{2}r_{11^{'}}^{2}+4\alpha_{1}(\alpha_{1}+2\alpha_{2})r_{11^{'}}+16\alpha_{1}h_{4}}-(2\alpha_{2}+3\alpha_{1}r_{11^{'}})}{\sqrt{32\alpha^{2}_{1}r_{11^{'}}^{2}+16\alpha_{1}(\alpha_{1}+\alpha_{2})r_{11^{'}}+64\alpha_{1}h_{4}}}\cdot
\sin^{2}\theta}}}.$
\section{Exact traveling wave solutions of Eq. (\ref{1})}\label{sec4}
\par In order to get the type I traveling wave solution of the GKMN equation, it only needs to substitute the solution $p(\xi)$ of system (\ref{8}) into the formula (\ref{6}). Here we list them in the Appendix for the sake of simplicity.
\par But, it is not so easy to get the type II traveling wave solution of the GKMN equation. It needs us to substitute the solutions $\phi(\xi)$ of system (\ref{13}) into the ODE  (\ref{12}) to solve  $\varphi(\xi)$, and then plug $\phi(\xi)$ and $\varphi(\xi)$ into the formula (\ref{10}).
\par\noindent (S1)
Noting that
\begin{equation*}
  \phi_{b_{1}}(\xi)=\sqrt{r_{1}+(r_{2}-r_{1})sn^{2}(\sqrt{\frac{\alpha_{1}(r_{3}-r_{1})}{2}}\xi)},~~-T_{0^{'}}<\xi<T_{0^{'}},\\
\end{equation*}
and
$$
\int{\frac{du}{1+k\cdot sn^{2}(u)}}=\frac{u}{2}+\frac{1}{2(1+k)}\tan^{-1}[(1+k)tn(u)\cdot nd(u)],
$$
we have
\begin{equation*}
\begin{aligned}
\varphi_{b_{1}}(\xi)
&=\int{\bigg(\displaystyle\frac{e}{am\phi_{b_{1}}^{2}(\xi)}+\displaystyle\frac{c}{2am}\bigg)d\xi}\\
&=\displaystyle\frac{e+cr_{1}}{2amr_{1}}\xi+\displaystyle\frac{r_{1}}{2r_{2}}\sqrt{\frac{2}{\alpha_{1}(r_{3}-r_{1})}}\cdot
\tan^{-1}\bigg(\frac{r_{2}}{r_{1}}tn(\sqrt{\frac{\alpha_{1}(r_{3}-r_{1})}{2}}\xi)\cdot nd(\sqrt{\frac{\alpha_{1}(r_{3}-r_{1})}{2}}\xi)\bigg)+C_{1},
\end{aligned}
\end{equation*}
where $C_{1}$ is a constant.
\par Thus we obtain the final solution $q_{1}(x,y,t)=\phi_{b_{1}}(\xi){\rm exp}(\varphi_{b_{1}}(\xi)-\mu t)i)$.
\par\noindent (S2)
Noting that
$$
\phi_{b_{2}}(\xi)=\sqrt{r_{4}+ \frac{(r_{5}-r_{4})(1-{\rm exp}(\sqrt{2\alpha_{1}(r_{5}-r_{4})}\xi))^{2}}{(1+{\rm exp}(\sqrt{2\alpha_{1}(r_{5}-r_{4})}\xi))^{2}}},~~-\infty<\xi<+\infty,\\
$$
we have
\begin{equation*}
\begin{aligned}
\varphi_{b_{2}}(\xi)
&=\int{\bigg(\displaystyle\frac{e}{am\phi_{b_{2}}^{2}(\xi)}+\displaystyle\frac{c}{2am}\bigg)d\xi}\\
&=\displaystyle\frac{2e\sqrt{2\alpha_{1}(r_{5}-r_{4})}+cr_{5}}{2amr_{5}}\xi+\displaystyle\frac{(r_{5}-r_{4})e}{8amr_{5}\sqrt{r_{4}(r_{5}-r_{4})}}\\
&\cdot
\arctan\bigg(\frac{r_{5}-r_{4}}{2\sqrt{r_{4}(r_{5}-r_{4})}}\exp(\sqrt{2\alpha_{1}(r_{5}-r_{4})}\xi)+\frac{r_{5}-r_{4}}{2\sqrt{r_{4}(r_{5}-r_{4})}}-\frac{r_{5}}{2r_{4}}\bigg)+C_{2},
\end{aligned}
\end{equation*}
where $C_{2}$ is a constant.
\par Thus we obtain the final solution $q_{2}(x,y,t)=\phi_{b_{2}}(\xi){\rm exp}(\varphi_{b_{2}}(\xi)-\mu t)i)$.
\par\noindent (S3)
Noting that
$$
\phi_{b_{3}}(\xi)=\sqrt{r_{6}+\displaystyle\frac{(2\alpha_{1}r_{6}+\alpha_{1}r_{7}+2\alpha_{2})(r_{7}-r_{6})sn^{2}(\sqrt{-\frac{
2\alpha_{1}r_{7}+\alpha_{1}r_{6}+2\alpha_{2}}{2}}\xi)}{\alpha_{1}(r_{7}-r_{6})
sn^{2}(\sqrt{-\frac{2\alpha_{1}r_{7}+\alpha_{1}r_{6}+2\alpha_{2}}{2}}\xi)-(2\alpha_{1}r_{7}+\alpha_{1}r_{6}+2\alpha_{2})}},~~-T_{1^{'}}<\xi<T_{1^{'}},
$$
and
$$
\int{\displaystyle\frac{du}{1\pm k\cdot sn(u)}}=\displaystyle\frac{1}{k^{'2}}[E(u)+k(1\pm k\cdot sn(u))\cdot cd(u)],
$$
where $k^{'}=\sqrt{1-k^{2}}$, we have
\begin{equation*}
\begin{aligned}
\varphi_{b_{3}}(\xi)
&=\int{\bigg(\displaystyle\frac{e}{am\phi_{b_{3}}^{2}(\xi)}+\displaystyle\frac{c}{2am}\bigg)d\xi}\\
&=\bigg(\frac{e\alpha_{1}}{am(3\alpha_{1}r_{6}+\alpha_{1}r_{7}+2\alpha_{2})}\sqrt{-\frac{2}{2\alpha_{1}r_{7}+\alpha_{1}r_{6}+2\alpha_{2}}}+\frac{c}{2am}\bigg)\xi\\
&+\frac{e(3\alpha_{1}r_{6}+\alpha_{1}r_{7}+2\alpha_{2}-\alpha_{1})\sqrt{-2(2\alpha_{1}r_{7}+\alpha_{1}r_{6}+2\alpha_{2})}}{am(3\alpha_{1}r_{6}+\alpha_{1}r_{7}
+2\alpha_{2})(2\alpha_{1}r_{6}+\alpha_{1}r_{7}+2\alpha_{2})}\cdot\bigg(E(\sqrt{-\frac{2\alpha_{1}r_{7}+\alpha_{1}r_{6}+2\alpha_{2}}{2}}\xi)\\
&+\sqrt {\frac{\alpha_{1}r_{7}-\alpha_{1}r_{6}}{2\alpha_{1}r_{7}+\alpha_{1}r_{6}+2\alpha_{2}}}\cdot
cd(\sqrt{-\frac{2\alpha_{1}r_{7}+\alpha_{1}r_{6}+2\alpha_{2}}{2}}\xi)\bigg)+C_{3},
\end{aligned}
\end{equation*}
where $C_{3}$ is a constant.
\par Thus we obtain the final solution $q_{3}(x,y,t)=\phi_{b_{3}}(\xi){\rm exp}(\varphi_{b_{3}}(\xi)-\mu t)i)$.
\par\noindent (S4)
Noting that
\begin{equation*}
  \phi_{u_{1}}(\xi)=\sqrt{r_{1^{'}}-\sqrt[4]{\frac{2\alpha_{1}r_{1^{'}}^{2}+(\alpha_{1}+2\alpha_{2})r_{1^{'}}+4h_{1}}{\alpha_{1}}}
  +\frac{2\sqrt[4]{\frac{2\alpha_{1}r_{1^{'}}^{2}+(\alpha_{1}+2\alpha_{2})r_{1^{'}}+4h_{1}}{\alpha_{1}}}}
  {1-cn(\sqrt{2\alpha_{1}}\sqrt[4]{\frac{2\alpha_{1}r_{1^{'}}^{2}+(\alpha_{1}+2\alpha_{2})r_{1^{'}}+4h_{1}}{\alpha_{1}}}\xi)}},\\
\end{equation*}
and
$$m
\int{\displaystyle\frac{1-cn(u)}{1+\alpha\cdot
cn(u)}du}=\displaystyle\frac{1}{\alpha}[-u+\displaystyle\frac{1}{1-\alpha}(\Pi(\varphi,\displaystyle\frac{\alpha^{2}}{\alpha^{2}-1},k)-\alpha\cdot f_{1})],
$$
where $0<\xi<\xi_{0^{''}}$, $\varphi=am(u)$ and
$
f_{1}=\sqrt{\frac{1-\alpha^{2}}{k^{2}+k^{'2}\alpha^{2}}}\cdot\tan^{-1}[\sqrt{\frac{k^{2}+k^{'2}\alpha^{2}}{1-\alpha^{2}}}\cdot sd(u)],~~k^{'}=\sqrt{1-k^{2}},
$
we have
\begin{equation*}
\begin{aligned}
\varphi_{u_{1}}(\xi)
&=\int{\bigg(\displaystyle\frac{e}{am\phi_{u_{1}}^{2}(\xi)}+\displaystyle\frac{c}{2am}\bigg)d\xi}\\
&=\bigg(\frac{e}{am(r_{1^{'}}-\sqrt[4]{\frac{2\alpha_{1}r_{1^{'}}^{2}+(\alpha_{1}+2\alpha_{2})r_{1^{'}}+4h_{1}}{\alpha_{1}}})}+\frac{c}{2am}\bigg)\xi\\
&+\frac{e(\sqrt[4]{\frac{2\alpha_{1}r_{1^{'}}^{2}+(\alpha_{1}+2\alpha_{2})r_{1^{'}}+4h_{1}}{\alpha_{1}}}+r_{1^{'}})}{2amr_{1^{'}}\sqrt[4]{8\alpha^{2}_{1}r_{1^{'}}^{2}+4\alpha_{1}(\alpha_{1}+2\alpha_{2})r_{1^{'}}+16\alpha_{1}h_{1}}(\sqrt[4]{\frac{2\alpha_{1}r_{1^{'}}^{2}+(\alpha_{1}+2\alpha_{2})r_{1^{'}}+4h_{1}}{\alpha_{1}}}-r_{1^{'}})}\\
&\cdot\Pi[am(\sqrt[4]{8\alpha^{2}_{1}r_{1^{'}}^{2}+4\alpha_{1}(\alpha_{1}+2\alpha_{2})r_{1^{'}}+16\alpha_{1}h_{1}}\xi),-\frac{(\sqrt[4]{\frac{2\alpha_{1}r_{1^{'}}^{2}+(\alpha_{1}+2\alpha_{2})r_{1^{'}}+4h_{1}}{\alpha_{1}}}-r_{1^{'}})^{2}}{4
\sqrt[4]{\frac{2\alpha_{1}r_{1^{'}}^{2}+(\alpha_{1}+2\alpha_{2})r_{1^{'}}+4h_{1}}{\alpha_{1}}}r_{1^{'}}},k]\\
&-\frac{e}{amr_{1^{'}}\sqrt[4]{8\alpha^{2}_{1}r_{1^{'}}^{2}+4\alpha_{1}(\alpha_{1}+2\alpha_{2})r_{1^{'}}+16\alpha_{1}h_{1}}}\\
&\cdot\sqrt{\frac{\sqrt[4]{\frac{2\alpha_{1}r_{1^{'}}^{2}+(\alpha_{1}+2\alpha_{2})r_{1^{'}}+4h_{1}}{\alpha_{1}}}r_{1^{'}}}{4r_{1^{'}}k^{2}\sqrt[4]{\frac{2\alpha_{1}r_{1^{'}}^{2}+(\alpha_{1}
+2\alpha_{2})r_{1^{'}}+4h_{1}}{\alpha_{1}}}+(\sqrt[4]{\frac{2\alpha_{1}r_{1^{'}}^{2}+(\alpha_{1}+2\alpha_{2})r_{1^{'}}+4h_{1}}{\alpha_{1}}}-r_{1^{'}})^{2}}}\\
&\cdot
\tan^{-1}\bigg(\sqrt{k^{2}+\frac{(\sqrt[4]{\frac{2\alpha_{1}r_{1^{'}}^{2}+(\alpha_{1}+2\alpha_{2})r_{1^{'}}+4h_{1}}{\alpha_{1}}}-r_{1^{'}})^{2}}{4\sqrt[4]{\frac{2\alpha_{1}r_{1^{'}}^{2}+(\alpha_{1}+2\alpha_{2})r_{1^{'}}+4h_{1}}{\alpha_{1}}}r_{1^{'}}}}\cdot\\
&sd(\sqrt[4]{8\alpha^{2}_{1}r_{1^{'}}^{2}+4\alpha_{1}(\alpha_{1}+2\alpha_{2})r_{1^{'}}+16\alpha_{1}h_{1}}\xi)\bigg)+C_{4},
\end{aligned}
\end{equation*}
where
$k^{2}=\displaystyle\frac{\sqrt{8\alpha_{1}^{2}r_{1^{'}}^{2}+4\alpha_{1}(\alpha_{1}+2\alpha_{2})r_{1^{'}}+16\alpha_{1}h_{1}}-(2\alpha_{2}+3\alpha_{1}r_{1^{'}})}{\sqrt{32\alpha^{2}_{1}r_{1^{'}}^{2}+16\alpha_{1}(\alpha_{1}+\alpha_{2})r_{1^{'}}+64\alpha_{1}h_{1}}}$,
$C_{4}$ is a constant.
\par Thus we obtain the final solution $q_{4}(x,y,t)=\phi_{u_{1}}(\xi){\rm exp}(\varphi_{u_{1}}(\xi)-\mu t)i)$. Similar calculation can be applied to the solutions $\phi_{u_{5}}(\xi)$, $\phi_{u_{7}}(\xi)$ and $\phi_{u_{8}}(\xi)$, we ignore them here for simplicity.
\par\noindent (S5)
Noting that
$$
\phi_{u_{2}}(\xi)=\sqrt{r_{2^{'}}+ \frac{(r_{3^{'}}-r_{2^{'}})(1+{\rm exp}(\sqrt{2\alpha_{1}(r_{3^{'}}-r_{2^{'}})}\xi))^{2}}{(1-{\rm
exp}(\sqrt{2\alpha_{1}(r_{3^{'}}-r_{2^{'}})}\xi))^{2}}},~~\xi>0,\\
$$
we have
\begin{equation*}
\begin{aligned}
\varphi_{u_{2}}(\xi)
&=\int{\bigg(\displaystyle\frac{e}{am\phi_{u_{2}}^{2}(\xi)}+\displaystyle\frac{c}{2am}\bigg)d\xi}\\
&=\frac{4e(r_{2^{'}}-r_{3^{'}})}{amr_{3^{'}}\sqrt{2\alpha_{1}(r_{3^{'}}-r_{2^{'}})}}\ln(r_{3^{'}}\sqrt{2\alpha_{1}(r_{3^{'}}-r_{2^{'}})}\xi+r_{3^{'}}-2r_{2^{'}})\\
&+\frac{4e(r_{2^{'}}-r_{3^{'}})^{2}}{amr_{3^{'}}\sqrt{2\alpha_{1}(r_{3^{'}}-r_{2^{'}})}(r_{3^{'}}\exp(\sqrt{2\alpha_{1}(r_{3^{'}}-r_{2^{'}})}\xi)+r_{3^{'}}-2r_{2^{'}})}+\frac{c}{2am}\xi+C_{5},
\end{aligned}
\end{equation*}
where $C_{5}$ is a constant.
\par Thus we obtain the final solution $q_{5}(x,y,t)=\phi_{u_{2}}(\xi){\rm exp}(\varphi_{u_{2}}(\xi)-\mu t)i)$.
\par\noindent (S6)
Noting that
\begin{equation*}
  \phi_{u_{3}}(\xi)=\sqrt{r_{4^{'}}+\frac{r_{6^{'}}-r_{4^{'}}}{sn^{2}(\sqrt{\displaystyle\frac{\alpha_{1}(r_{6^{'}}-r_{4^{'}})}{2}}\xi)}},~~0<\xi<\xi_{1^{''}},\\
\end{equation*}
and
$$
\int{\frac{du}{1+k\cdot sn^{2}(u)}}=\frac{u}{2}+\frac{1}{2(1+k)}\cdot\tan^{-1}[(1+k)\cdot tn(u)\cdot nd(u)],
$$
we have
\begin{equation*}
\begin{aligned}
\varphi_{u_{3}}(\xi)
&=\int{\bigg(\displaystyle\frac{e}{am\phi_{u_{3}}^{2}(\xi)}+\displaystyle\frac{c}{2am}\bigg)d\xi}\\
&=\frac{e+cr_{4^{'}}}{2amr_{4^{'}}}\xi-\frac{e(r_{6^{'}}-r_{4^{'}})}{4amr_{4^{'}}^{2}r_{6^{'}}}\sqrt{\frac{2}{\alpha_{1}(r_{6^{'}}-r_{4^{'}})}}\\
&\cdot\tan^{-1}\bigg(\frac{r_{6^{'}}}{r_{6^{'}}-r_{4^{'}}}\cdot tn(\sqrt{\frac{\alpha_{1}(r_{6^{'}}-r_{4^{'}})}{2}}\xi)\cdot
nd(\sqrt{\frac{\alpha_{1}(r_{6^{'}}-r_{4^{'}})}{2}}\xi)\bigg)+C_{6},
\end{aligned}
\end{equation*}
where $C_{6}$ is a constant.
\par Thus we obtain the final solution $q_{6}(x,y,t)=\phi_{u_{3}}(\xi){\rm exp}(\varphi_{u_{3}}(\xi)-\mu t)i)$.
\par\noindent (S7)
Noting that
$$
\phi_{u_{4}}(\xi)=\sqrt{r_{8^{'}}+(r_{8^{'}}-r_{7^{'}})\cdot \cot^{2}(\sqrt{\frac{\alpha_{1}(r_{8^{'}}-r_{7^{'}})}{2}}\xi)} ,~~0<\xi<\xi_{2^{''}},
$$
we have
\begin{equation*}
\begin{aligned}
\varphi_{u_{4}}(\xi)
&=\int{\bigg(\displaystyle\frac{e}{am\phi_{u_{4}}^{2}(\xi)}+\displaystyle\frac{c}{2am}\bigg)d\xi}\\
&=\frac{2e+cr_{7^{'}}}{2amr_{7^{'}}}\xi-\frac{e}{amr_{7^{'}}}\sqrt{\frac{2}{\alpha_{1}r_{8^{'}}}}\cdot\arctan\bigg(\sqrt{\frac{r_{8^{'}}}{r_{8^{'}}-r_{7^{'}}}}\cdot
\tan(\sqrt{\frac{\alpha_{1}(r_{8^{'}}-r_{7^{'}})}{2}}\xi)\bigg)+C_{7},
\end{aligned}
\end{equation*}
where $C_{7}$ is a constant.
\par Thus we obtain the final solution $q_{7}(x,y,t)=\phi_{u_{4}}(\xi){\rm exp}(\varphi_{u_{4}}(\xi)-\mu t)i)$.
\par\noindent (S8)
Noting that
\begin{equation*}
  \phi_{u_{6}}(\xi)=\sqrt{-\frac{2\alpha_{2}}{3\alpha_{1}}+\frac{2}{\alpha_{1}\xi^{2}}},~~\xi>0,\\
\end{equation*}
we have
\begin{equation*}
\begin{aligned}
\varphi_{u_{6}}(\xi)
&=\int{\bigg(\displaystyle\frac{e}{am\phi_{u_{6}}^{2}(\xi)}+\displaystyle\frac{c}{2am}\bigg)d\xi}\\
&=\frac{c\alpha_{2}-3e\alpha_{1}}{2am\alpha_{2}}\xi-\frac{9e\alpha_{1}^{2}}{4am\alpha_{2}^{2}}\sqrt{-\frac{3}{\alpha_{2}}}\cdot \arctan(\sqrt{-\frac{\alpha_{2}}{3}}\xi)+C_{8},
\end{aligned}
\end{equation*}
where $C_{8}$ is a constant.
\par Thus we obtain the final solution $q_{8}(x,y,t)=\phi_{u_{6}}(\xi){\rm exp}(\varphi_{u_{6}}(\xi)-\mu t)i)$.
\section{Discussion and conclusion}\label{sec5}
\par In this paper, by using the dynamical system method, we study two kinds of traveling wave systems of the GKMN equation and  obtain all type I and type II traveling wave solutions of it. Especially, some new solutions $q_{b_{7}}$, $q_{u_{\iota}}(\iota=2,2^{'},4..8,4^{'}..8^{'})$ and $q_{j}(j=1..8)$ have not been reported before, which not only help ones to understand the complicated physical phenomena described by the model further, but also can be used to verify the correctness of the numerical solutions. In particular, we can generate more  solutions of Eq. (\ref{1}) by using of these new solutions. For example, when $a=-1$ and $b=-2$, one can consider these new solutions as "seeds" and apply the perturbation $(n;M)$-fold Darboux transformation $\widetilde{\varphi}=T(\lambda)\varphi$ and $\widetilde{q}_{_{N-1}}=q_{0}-2B^{^{(N-1)}}$ mentioned in \cite{t3e}  to  construct more new solutions. In addition, this method is an effective way to deal with  traveling waves of a PDE and can be used in other PDE models.
\section*{Acknowledgements}
\par This work is supported by the Natural Science Foundation of China (No.11301043) and China Postdoctoral Science Foundation (No.2016M602663).
\section*{References}
\bibliography{ref}
\section*{Appendix}
\begin{table}[H]\tiny
\caption{Type I traveling wave solutions of Eq. (\ref{1}) }
\centering
\resizebox{12cm}{8.5cm}{
\begin{tabular}{|c|c|c|c|c|}
\hline
$ \displaystyle \kappa b(a\kappa \omega-r)$ & $\displaystyle\frac{\kappa b}{am}$ &\makecell[c]{traveling wave solution} & $\makecell[c]{p}$ & $\xi$\\
\hline
 \multirow{16}{*}{$\makecell[c]{\displaystyle \kappa b(a\kappa \omega-r)>0}$} & \multirow{11}{*}{$\makecell[c]{\displaystyle\frac{\kappa b}{am}<0}$}
  &   $\makecell[l] {q_{b_{1}}(x,y,t)=\Bigg(p_{1}+\frac{(p_{2}-p_{1})(p_{3}-p_{1})}{(p_{3}-p_{1})-(p_{3}-p_{2})sn^{2}(\sqrt{\frac{-\kappa b(p_{4}-p_{2})(p_{3}-p_{1})}{4am}}\xi)}\Bigg)\cdot{\rm
\exp}((\kappa x+\omega y-rt)i)}$
   &    $p_{1}<-\sqrt{\frac{a\kappa \omega-r}{2\kappa b}}<p_{2}<p<p_{3}<\sqrt{\frac{a\kappa \omega-r}{2\kappa b}}<p_{4}$ & $-T_{0}<\xi<T_{0}$ \\
 & & $\makecell[l]{q_{b_{2}}(x,y,t)=\Bigg(\frac{p_{6}-p_{5}}{2}\tanh(\frac{p_{6}-p_{5}}{2}\sqrt{-\frac{\kappa b}{am}}\xi)\Bigg)\cdot{\rm \exp}((\kappa x+\omega y-rt)i)
\\ q_{b_{2^{'}}}(x,y,t)=\Bigg(\frac{p_{6}-p_{5}}{2}\tanh(\frac{p_{5}-p_{6}}{2}\sqrt{-\frac{\kappa b}{am}}\xi)\Bigg)\cdot{\rm \exp}((\kappa x+\omega y-rt)i)}$ & $-\sqrt{\displaystyle \frac{a\kappa \omega-r}{2\kappa
b}}=p_{5}<p<p_{6}=\sqrt{\displaystyle\frac{a\kappa \omega-r}{2\kappa b}}$ &$-\infty<\xi<+\infty$\\

& & $\makecell[l]{q_{u_{0}}(x,y,t)=\sqrt[4]{-\displaystyle\frac{2amh_{0}}{\kappa b}}\cdot\sqrt{-1+\frac{2}{1-cn(2\sqrt[4]{-\frac{2\kappa bh_{0}}{am}}\xi)}}\cdot{\rm \exp}((\kappa x+\omega y-rt)i)}$ & $0<p<+\infty$ &$0<\xi<\xi_{0}$
 \\
& & $\makecell[l]{q_{u_{1}}(x,y,t)=\sqrt{\displaystyle\frac{a\kappa \omega-r}{2\kappa b}}\cdot\bigg(1+\frac{2}{\exp(\sqrt{\frac{2(r-a\kappa\omega)}{am}}\xi)-1}\bigg)\cdot{\rm \exp}((\kappa x+\omega y-rt)i)}$ & $0<\sqrt{\displaystyle\frac{a\kappa \omega-r}{2\kappa b}}<p<+\infty$ &$\xi>0$
 \\
& &  $\makecell[l]{q_{u_{{1}^{'}}}(t,x,y)=-\sqrt{\displaystyle\frac{a\kappa \omega-r}{2\kappa b}}\cdot\bigg(1+\frac{2}{\exp(\sqrt{\frac{2(r-a\kappa\omega)}{am}}\xi)-1}\bigg)\cdot{\rm \exp}((\kappa x+\omega y-rt)i)}$ & $-\infty<p<-\sqrt{\displaystyle \frac{a\kappa \omega-r}{2\kappa b}}<0$ &$\xi>0$
\\
& & $\makecell[l]{q_{u_{{2}}}(x,y,t)=\displaystyle\frac{p_{4^{'}}}{sn(p_{4^{'}}\sqrt{-\frac{\kappa b}{am}}\xi)}\cdot{\rm \exp}((\kappa x+\omega y-rt)i)}$
&$p_{1^{'}}<-\sqrt{\displaystyle \frac{a\kappa \omega-r}{2\kappa b}}<p_{2^{'}} <0<p_{3^{'}}<\sqrt{\displaystyle \frac{a\kappa \omega-r}{2\kappa b}}<p_{4^{'}}<p<+\infty$ & $0<\xi<\xi_{1}$
\\
& & $\makecell[l]{q_{u_{{2}^{'}}}(x,y,t)=\displaystyle\frac{p_{1^{'}}}{sn(p_{1^{'}}\sqrt{-\frac{\kappa b}{am}}\xi)}\cdot{\rm \exp}((\kappa x+\omega y-rt)i)}$
&$-\infty<p<p_{1^{'}}<-\sqrt{\frac{a\kappa \omega-r}{2\kappa b}}<p_{2^{'}}<0<p_{3^{'}}<\sqrt{\frac{a\kappa \omega-r}{2\kappa b}}<p_{4^{'}}$ & $0<\xi<\xi_{1^{'}}$
\\
& & $\makecell[l]{q_{u_{{3}}}(x,y,t)=p_{5^{'}}\csc(p_{5^{'}}\sqrt{-\frac{\kappa b}{am}}\xi)\cdot{\rm \exp}((\kappa x+\omega y-rt)i)}$ & $0<\sqrt{\frac{a\kappa \omega-r}{2\kappa b}}<p_{5^{'}}<p<+\infty$ &$0<\xi<\xi_{2}$
\\
& & $\makecell[l]{q_{u_{{3}^{'}}}(x,y,t)=p_{5^{'}}\csc(-p_{5^{'}}\sqrt{-\frac{\kappa b}{am}}\xi)\cdot{\rm \exp}((\kappa x+\omega y-rt)i)}$ & $-\infty<p<p_{5^{'}}<-\sqrt{\frac{a\kappa \omega-r}{2\kappa b}}<0$ &$0<\xi<\xi_{2^{'}}$
\\
& & $\makecell[l]{q_{u_{{4}}}(x,y,t)=\sqrt{-\displaystyle\frac{\kappa bp_{6^{'}}^{2}-a\kappa\omega+r}{\kappa b}+\frac{2\kappa bp_{6^{'}}^{2}-a\kappa\omega+r}{\kappa bsn^{2}(\sqrt{-\frac{2\kappa
bp_{6^{'}}^{2}-a\kappa\omega+r}{am}}\xi)}}\cdot{\rm \exp}((\kappa x+\omega y-rt)i)}$ & $0<\sqrt{\displaystyle\frac{a\kappa\omega-r}{\kappa b}}<p_{6^{'}}<p<+\infty$ &$0<\xi<\xi_{3}$
\\
& & $\makecell[l]{q_{u_{{4}^{'}}}(x,y,t)=-\sqrt{-\displaystyle\frac{\kappa bp_{6^{'}}^{2}-a\kappa\omega+r}{\kappa b}+\frac{2\kappa bp_{6^{'}}^{2}-a\kappa\omega+r}{\kappa bsn^{2}(\sqrt{-\frac{2\kappa
bp_{6^{'}}^{2}-a\kappa\omega+r}{am}}\xi)}}\cdot{\rm \exp}((\kappa x+\omega y-rt)i)}$ & $-\infty<p<-p_{6^{'}}<-\sqrt{\displaystyle\frac{a\kappa\omega-r}{\kappa b}}<0$ &$0<\xi<\xi_{3^{'}}$
\\
\cline{2-5}
& \multirow{5}{*}{$\makecell[c]{\displaystyle\frac{\kappa b}{am}>0}$}
  &   $\makecell[l] {q_{b_{3}}(x,y,t)=\Bigg(p_{10}-\frac{(p_{10}-p_{8})(p_{10}-p_{7})}{(p_{10}-p_{8})+(p_{8}-p_{7})sn^{2}(\sqrt{\frac{\kappa
b}{am}}\frac{\sqrt{(p_{10}-p_{8})(p_{9}-p_{7})}}{2}\xi)}\Bigg)\cdot{\rm \exp}((\kappa x+\omega y-rt)i)}$ & $p_{7}<p<p_{8}<0<p_{9}<\sqrt{\frac{a\kappa \omega-r}{2\kappa b}}<p_{10}$ &$-T_{1}<\xi<T_{1}$\\
& & $\makecell[l]{q_{b_{3^{'}}}(x,y,t)=\Bigg(p_{8}+\frac{(p_{9}-p_{8})(p_{10}-p_{8})}{(p_{10}-p_{8})-(p_{10}-p_{9})sn^{2}(\sqrt{\frac{\kappa
b}{am}}\frac{\sqrt{(p_{10}-p_{8})(p_{9}-p_{7})}}{2}\xi)}\Bigg)\cdot{\rm \exp}((\kappa x+\omega y-rt)i)}$ & $p_{7}<-\sqrt{\frac{a\kappa \omega-r}{2\kappa b}}<p_{8}<0<p_{9}<p<p_{10}$ &$-T_{1}<\xi<T_{1}$\\
& & $\makecell[l]{q_{b_{4}}(x,y,t)=\Bigg(\frac{-2p_{11}{\rm exp}(\sqrt{\frac{\kappa b}{am}}p_{11}\mid\xi\mid)}{{{\rm \exp}(2\sqrt{\frac{\kappa b}{am}}p_{11}\mid\xi\mid)}+1}\Bigg)\cdot{\rm \exp}((\kappa x+\omega y-rt)i)}$ &  $-p_{11}<p<0<\sqrt{\frac{a\kappa \omega-r}{2\kappa b}}<p_{11}$ &$-\infty<\xi<+\infty$
 \\
& & $\makecell[l]{q_{b_{4^{'}}}(x,y,t)=\Bigg(\frac{2p_{11}{\rm \exp}(\sqrt{\frac{\kappa b}{am}}p_{11}\mid\xi\mid)}{{{\rm \exp}(2\sqrt{\frac{\kappa b}{am}}p_{11}\mid\xi\mid)}+1}\Bigg)\cdot{\rm exp}((\kappa x+\omega y-rt)i)}$
  & $-p_{11}<-\sqrt{\frac{a\kappa \omega-r}{2\kappa b}}<0<p<p_{11}$ &$-\infty<\xi<+\infty$
\\
& & $\makecell[l]{q_{b_{5}}(x,y,t)=\sqrt{\displaystyle\frac{(\kappa bp_{12}^{2}-a\kappa\omega+r)p_{12}^{2}sn^{2}(\sqrt{\frac{2\kappa bp_{12}^{2}-a\kappa\omega+r}{am}}\xi)}{2\kappa
bp_{12}^{2}-a\kappa\omega+r-\kappa bp_{12}^{2}sn^{2}(\sqrt{\frac{2\kappa bp_{12}^{2}-a\kappa\omega+r}{am}}\xi)}}\cdot{\rm \exp}((\kappa x+\omega y-rt)i) }$
  & $-p_{12}<p<p_{12}$  &$-T_{2}<\xi<T_{2}$
\\
\hline
\multirow{5}{*}{$\makecell[c]{\displaystyle \kappa b(a\kappa \omega-r)}<0$} & \multirow{1}{*}{$\makecell[c]{\displaystyle\frac{\kappa b}{am}<0}$}
  &   $\makecell[l] {q_{u_{{5}}}(x,y,t)=\sqrt[4]{-\displaystyle\frac{2amh_{0^{'}}}{\kappa b}}\cdot\sqrt{-1+\frac{2}{1-cn(2\sqrt[4]{-\frac{2\kappa bh_{0^{'}}}{am}}\xi)}}\cdot{\rm \exp}((\kappa x+\omega y-rt)i)}$ & $0<p<+\infty$ &$0<\xi<\xi_{4}$\\
  & & $\makecell[l] {q_{u_{{6}}}(x,y,t)=2\sqrt{\frac{r-a\kappa\omega}{\kappa b}}\cdot\displaystyle\frac{\exp(\sqrt{\frac{a\kappa\omega-r}{am}}\xi)}{\exp(2\sqrt{\frac{a\kappa\omega-r}{am}}\xi)-1}\cdot{\rm \exp}((\kappa x+\omega y-rt)i)}$ & $0<p<+\infty$ &$\xi>0$\\
 & & $\makecell[l]{q_{u_{{6}^{'}}}(x,y,t)=2\sqrt{\frac{r-a\kappa\omega}{\kappa b}}\cdot\displaystyle\frac{\exp(\sqrt{\frac{a\kappa\omega-r}{am}}\xi)}{1-\exp(2\sqrt{\frac{a\kappa\omega-r}{am}}\xi)}\cdot{\rm \exp}((\kappa x+\omega y-rt)i)}$ & $-\infty<p<0$ &$~\xi>0$\\
& & $\makecell[l]{q_{u_{{8}}}(x,y,t)=\sqrt{-\displaystyle\frac{\kappa bp_{7^{'}}^{2}-a\kappa\omega+r}{\kappa b}+\frac{2\kappa bp_{7^{'}}^{2}-a\kappa\omega+r}{\kappa bsn^{2}(\sqrt{-\frac{2\kappa
bp_{7^{'}}^{2}-a\kappa\omega+r}{am}}\xi)}}\cdot{\rm \exp}((\kappa x+\omega y-rt)i)}$ & $0<p_{7^{'}}<p<+\infty$ &$0<\xi<\xi_{5}$
 \\
& & $\makecell[l]{q_{u_{{8^{'}}}}(x,y,t)=-\sqrt{-\displaystyle\frac{\kappa bp_{7^{'}}^{2}-a\kappa\omega+r}{\kappa b}+\frac{2\kappa bp_{7^{'}}^{2}-a\kappa\omega+r}{\kappa bsn^{2}(\sqrt{-\frac{2\kappa
bp_{7^{'}}^{2}-a\kappa\omega+r}{am}}\xi)}}\cdot{\rm \exp}((\kappa x+\omega y-rt)i)}$ & $-\infty<p<-p_{7^{'}}<0$ &$0<\xi<\xi_{5^{'}}$
 \\
\cline{2-5}
& \multirow{2}{*}{$\makecell[c]{\displaystyle\frac{\kappa b}{am}>0}$}     & $\makecell[l] {q_{b_{6}}(x,y,t)=\sqrt{\displaystyle\frac{(\kappa bp_{13}^{2}-a\kappa\omega+r)p_{13}^{2}sn^{2}(\sqrt{\frac{2\kappa bp_{13}^{2}-a\kappa\omega+r}{am}}\xi)}{2\kappa
bp_{13}^{2}-a\kappa\omega+r-\kappa bp_{13}^{2}sn^{2}(\sqrt{\frac{2\kappa bp_{13}^{2}-a\kappa\omega+r}{am}}\xi)}}\cdot{\rm \exp}((\kappa x+\omega y-rt)i) }$
  & $-p_{13}<p<p_{13}$  &$-T_{3}<\xi<T_{3}$\\
\hline
\multirow{5}{*}{$\makecell[c]{\displaystyle \kappa b(a\kappa \omega-r)=0}$} & $\makecell[c]{\displaystyle\frac{\kappa b}{am}<0}$
  &   $\makecell[l] {q_{u_{{5^{'}}}}(x,y,t)=\sqrt{\displaystyle\frac{2}{1-cn(8h_{0^{'}}\xi)}-1}\cdot{\rm \exp}((\kappa x+\omega y-rt)i)}$
   &     $0<p<+\infty$ & $0<\xi<\xi_{4^{'}}$ \\
 & & $\makecell[l]
  {q_{u_{{7}}}(x,y,t)=\sqrt{-\frac{am}{\kappa b}}\cdot\frac{1}{\xi}\cdot{\rm \exp}((\kappa x+\omega y-rt)i)}$
   &     $0<p<+\infty$ & $\xi>0$ \\
 & & $\makecell[l]{q_{u_{{7^{'}}}}(x,y,t)=-\sqrt{-\frac{am}{\kappa b}}\cdot\frac{1}{\xi}\cdot{\rm \exp}((\kappa x+\omega y-rt)i)}$ & $-\infty<p<0$ &$\xi>0$\\
& & $\makecell[l]{q_{u_{{9}}}(x,y,t)=\sqrt{-\displaystyle\frac{\kappa bp_{8^{'}}^{2}}{\kappa b}+\frac{2\kappa bp_{8^{'}}^{2}}{\kappa bsn^{2}(\sqrt{-\frac{2\kappa
bp_{8^{'}}^{2}}{am}}\xi)}}\cdot{\rm \exp}((\kappa x+\omega y-rt)i)}$ & $0<p_{8^{'}}<p<+\infty$ &$0<\xi<\xi_{6}$
 \\
& & $\makecell[l]{q_{u_{{9^{'}}}}(x,y,t)=-\sqrt{-\displaystyle\frac{\kappa bp_{8^{'}}^{2}}{\kappa b}+\frac{2\kappa bp_{8^{'}}^{2}}{\kappa bsn^{2}(\sqrt{-\frac{2\kappa
bp_{8^{'}}^{2}}{am}}\xi)}}\cdot{\rm \exp}((\kappa x+\omega y-rt)i)}$ & $-\infty<p<-p_{8^{'}}<0$ &$0<\xi<\xi_{6^{'}}$
 \\
\cline{2-5}
& $\makecell[c]{\displaystyle\frac{\kappa b}{am}>0}$
  &   $\makecell[l] {q_{b_{7}}(x,y,t)=\sqrt{\displaystyle\frac{p_{14}^{2}sn^{2}(\sqrt{\frac{2\kappa bp_{14}^{2}}{am}}\xi)}{2-sn^{2}(\sqrt{\frac{2\kappa bp_{14}^{2}}{am}}\xi)}}\cdot{\rm \exp}((\kappa x+\omega y-rt)i) }$
  & $-p_{14}<p<p_{14}$  &$-T_{4}<\xi<T_{4}$\\
\hline
\end{tabular}}
\end{table}
\end{document}